\documentclass[a4paper]{amsart}

\usepackage{amsmath, amssymb, latexsym, amsfonts, pigpen, enumitem, mathrsfs}
\usepackage[matrix,arrow,curve]{xy}

\numberwithin{equation}{section}
\theoremstyle{plain}
\newtheorem{theorem}[equation]{Theorem}
\newtheorem{corollary}[equation]{Corollary}
\newtheorem{lemma}[equation]{Lemma}
\theoremstyle{definition}
\newtheorem{definition}[equation]{Definition}

\theoremstyle{remark}
\newtheorem{remark}[equation]{Remark}

\newcommand{\Proj}{\mathbb{P}}
\newcommand{\Hom}{{\mathrm{Hom}}}

\newcommand{\Z}{{\mathbb{Z}}}

\newcommand{\GW}{{\mathrm{GW}}}

\newcommand{\K}{\mathrm{K}}
\newcommand{\MW}{\mathrm{MW}}

\newcommand{\A}{\mathbb{A}}
\newcommand{\SSp}{\mathbb{S}}
\newcommand{\PP}{\mathbb{P}}
\newcommand{\NN}{\mathbb{N}}
\newcommand{\rank}{\operatorname{rank}}

\newcommand{\id}{\operatorname{id}}

\newcommand{\struct}{\mathcal{O}}
\newcommand{\Gm}{{\mathbb{G}_m}}
\newcommand{\GL}{\mathrm{GL}}
\newcommand{\SH}{\mathcal{SH}}

\newcommand{\T}{\mathrm{T}}

\newcommand{\Spec}{\operatorname{Spec}}

\newcommand{\coh}{\mathrm{H}}

\newcommand{\Zar}{\mathrm{Zar}}

\newcommand{\Sm}{\mathrm{Sm}}
\newcommand{\Sch}{\mathrm{Sch}}
\newcommand{\Fr}{\mathrm{Fr}}
\newcommand{\lF}{\mathrm{F}}
\newcommand{\ZF}{\mathbb{Z}\mathrm{F}}
\newcommand{\ZFr}{\mathbb{Z}\mathrm{Fr}}

\newcommand{\Nis}{\mathrm{Nis}}
\newcommand{\red}{\mathrm{red}}
\newcommand{\frp}{\star}

\newcommand{\cL}{\mathcal{L}}

\newcommand{\cU}{\mathcal{U}}

\newcommand{\relpcurve}{\overline{\mathscr{C}}}
\newcommand{\relcurve}{\mathscr{C}}
\newcommand{\et}{\mathrm{\acute{e}t}}
\newcommand{\fin}{\mathrm{fin}}
\newcommand{\cl}{\mathrm{cl}}
\newcommand{\Q}{\mathrm{Q}}
\newcommand{\Jac}{\operatorname{Jac}}
\newcommand{\cF}{\mathcal{F}}
\newcommand{\op}{\mathrm{op}}
\newcommand{\Pic}{\operatorname{Pic}}

\newcommand{\divisor}{\operatorname{div}}
\newcommand{\etale}{\'etale }
\newcommand{\Etale}{\'Etale }
\newcommand{\la}{\langle}
\newcommand{\ra}{\rangle}
\newcommand{\mm}{\mathfrak{m}}
\newcommand{\mmod}{\,\operatorname{mod}\,}
\newcommand{\chark}{\operatorname{char}}

\newcommand{\simA}{\stackrel{\A^1}{\sim}}

\newcommand{\pour}{\ar@{}[ur]|(0.2){\text{\pigpenfont G}}}
\newcommand{\podr}{\ar@{}[dr]|(0.2){\text{\pigpenfont A}}}

\begin{document}

\title[Rigidity for linear framed presheaves and motivic cohomology theories]{Rigidity for linear framed presheaves and generalized motivic cohomology theories}
\author{Alexey Ananyevskiy}
\address{St. Petersburg Department, Steklov Math. Institute, Fontanka 27, St. Petersburg 191023 Russia, and Chebyshev Laboratory, St. Petersburg State University, 14th Line V.O., 29B, St. Petersburg 199178 Russia}
\email{alseang@gmail.com}
\author{Andrei Druzhinin}
\address{Chebyshev Laboratory, St. Petersburg State University, 14th Line V.O., 29B, St. Petersburg 199178 Russia}
\email{andrei.druzh@gmail.com}
\thanks{Research is supported by the Russian Science Foundation grant 14-21-00035}

\begin{abstract}
A rigidity property for the homotopy invariant stable linear framed presheaves is established. As a consequence a variant of Gabber rigidity theorem is obtained for a cohomology theory representable in the motivic stable homotopy category by a $\phi$-torsion spectrum with $\phi\in\GW(k)$ of rank coprime to the (exponential) characteristic of the base field $k$. It is shown that the values of such cohomology theories at an essentially smooth Henselian ring and its residue field coincide. The result is applicable to cohomology theories representable by $n$-torsion spectra as well as to the ones representable by $\eta$-periodic spectra and spectra related to Witt groups.
\end{abstract}

\maketitle

\section{Introduction}

In the classical topology every generalized cohomology theory $E^*$ is locally constant, i.e. for a locally contractible space $X$ and a point $x\in X$ there exists a neighborhood $U$ of $x$ in $X$ such that the restriction to $\{x\}$ gives an isomorphism $E^*(U)\xrightarrow{\simeq} E^*(x)$. Considering the limit along the neighborhoods $U_\alpha$ of $x$ one obtains an isomorphism
\[
\varinjlim E^*(U_\alpha) \xrightarrow{\simeq} E^*(x)
\]
claiming that $E^*$ is infinitesimally constant. As one can see, this property is a direct consequence of the homotopy invariance of $E^*$. 

To the contrast, in the algebraic geometry most of the cohomology theories are not infinitesimally constant, although they enjoy the homotopy invariance property. An immediate example is given by the first Quillen $\K$-functor $\K_1^\Q$: the infinitesimal value is given by the units in the corresponding local ring and this group is much bigger than the group of units of the residue field. Studying this example further one notices that if the local ring is Henselian then the kernel of the restriction map is $n$-divisible for every $n$ prime to the characteristic of the residue field (abusing the notation we say that $n$ is prime to $\operatorname{char} k$ if $n$ is invertible in $k$, i.e. $n$ is prime to the exponential characteristic of $k$). It follows that $\K_1^\Q(-,\Z/n)$ is infinitesimally constant in the \etale topology. It turns out \cite[Theorem~2]{Gab92} that the higher Quillen $\K$-theory enjoys this property, i.e. the restriction induces an isomorphism
\[
\K_*^\Q(\struct_{X,x}^h,\Z/n)\xrightarrow{\simeq} \K_*^\Q(k,\Z/n)
\]
for a smooth variety $X$ over a field $k$, a rational point $x\in X$ and an integer $n$ prime to $\operatorname{char} k$. This result was obtained by Gabber soon after Suslin proved the theorem claiming that for an extension of algebraically closed fields $F_2/F_1$ and $n$ prime to $\operatorname{char}F_1$ one has
\[
\K_*^\Q(F_1,\Z/n)\xrightarrow{\simeq} \K_*^\Q(F_2,\Z/n),
\]
\cite[Main theorem]{Sus83}. Both the results are usually referred to as \textit{rigidity for $K$-theory}. At almost the same time Gillet and Thomason obtained a variant of this rigidity property for a strictly Henselian ring \cite[Theorem~A]{GT84}.

Analyzing the proofs one sees that the crucial properties of $\K^\Q_*$ used are homotopy invariance and the existence of transfers which allow one to construct a certain pairing with the (relative) Picard group. A similar strategy was later realized to obtain rigidity statements for different functors of cohomological nature: 
\begin{enumerate}
\item
Suslin and Voevodsky proved an analog of Gabber rigidity theorem for $n$-torsion homotopy invariant presheaves with transfers \cite[Theorem~4.4]{SV96};
\item
Panin and Yagunov obtained a version of Suslin rigidity theorem for $n$-torsion orientable functors \cite{PY02};
\item
Yagunov proved an analogue of Suslin rigidity theorem for a cohomology theory representable in the motivic stable homotopy category by an $n$-torsion spectrum \cite{Ya04}, R\"ondigs and \O stvaer obtained a categorical version of Yagunov's result \cite{RO08};
\item
Hornbostel and Yagunov proved an analogue of Gabber rigidity theorem for a cohomology theory representable in the motivic stable homotopy category by an $n$-torsion spectrum assuming that the action of $\GW(k)\cong \Hom_{\SH(k)}(\SSp,\SSp)$ on the cohomology theory factors through $\Z$ \cite{HY07};
\item 
Stavrova obtained an analogue of Gabber rigidity theorem for non-stable $\K_1$-functors of type $D_l$ satisfying a certain isotropy condition \cite[Corollary~1.4]{St14}.
\end{enumerate}
Based on the result by Suslin and Voevodsky, Morel derived a version of Gabber rigidity theorem for an $n$-torsion strictly $\A^1$-invariant sheaf with generalized transfers \cite[Theorem~5.14]{Mor11} (see loc. cit. for definitions) assuming a certain finiteness condition on the virtual cohomological $2$-dimension. A particular example of such sheaf is given by a homotopy sheaf of an $n$-torsion spectrum from the motivic stable homotopy category, i.e. Morel proved a version of Gabber rigidity theorem for a cohomology theory representable in the motivic stable homotopy category by an $n$-torsion spectrum provided that the base field satisfies a certain assumption on the finiteness of the virtual cohomological $2$-dimension. The latest rigidity result was obtained by Bachmann~\cite[Corollary~40]{Ba16} as a corollary of his study of $\rho$-inverted stable motivic homotopy category by means of real \etale topology. Bachmann showed that a version of Gabber rigidity theorem holds for a cohomology theory representable by a $\rho$-periodic spectrum with $\rho=-[-1]\in \K_1^{\MW}(k)\cong \Hom_{\SH(k)}(\SSp,\SSp\wedge (\A^1-\{0\},1))$. Note that there is a relation $h\rho=0$ for $h=\langle 1 \rangle + \langle -1 \rangle \in \GW(k)\cong \Hom_{\SH(k)}(\SSp,\SSp)$ whence a $\rho$-periodic spectrum is $h$-torsion.

All the above functors possess some kind of transfers, from the transfers given by correspondences introduced by Suslin and Voevodsky to the generalized transfers introduced by Morel. It is a remarkable observation due to Voevodsky \cite{V01} that was recently vastly developed by Garkusha and Panin \cite{GP14,GP15,AGP16,GNP16} that the existence of some kind of transfers is not that much restrictive: every $(\PP^1,\T)$-stable functor admits transfers along so-called framed correspondences. Here a $(\PP^1,\T)$-stable functor is a presheaf $\cF$ of abelian groups on the category of smooth schemes such that
\begin{enumerate}
\item
for every morphism of pointed Nisnevich sheaves $f\colon X_+\wedge (\PP^1,\infty)^{\wedge m} \to Y_+\wedge \T^{\wedge m}$ there is a homomorphism $f^*\colon \cF(Y)\to \cF (X)$ satisfying a certain naturality condition;
\item
for $f$ as above and for the morphism of pointed Nisnevich sheaves 
\[
\gamma\colon (\PP^1,\infty)\to \PP^1/(\PP^1-\{0\})\cong \A^1/(\A^1-\{0\}) =\T
\]
given by contraction and excision one has $(f\wedge \gamma)^*=f^*\colon \cF(Y)\to \cF (X)$.
\end{enumerate}
See Definition~\ref{def:PT} for the details. The choice of $(\PP^1,\T)$ (preferred over $(\PP^1,\PP^1)$, $(\T,\T)$ and ($\T,\PP^1$)) could seem to be an arbitrary one, but it is the one that allows to introduce framed correspondences to the picture, see Definition~\ref{def:FrPT}.

In the present paper we obtain the following result (see Sections~\ref{sect:def} and~\ref{sect:frpr} for the notation) generalizing all the discussed above rigidity statements. 
\begin{theorem}[Theorem~\ref{thm:rigidity}]\label{thm:intr_rigidity}
Let $k$ be a field, $X$ be a smooth variety over $k$ and $x\in X$ be a closed point such that $k(x)/k$ is separable. Let $\cF$ be a homotopy invariant stable linear framed presheaf over $k$ and $n\in\NN$ be invertible in $k$. Suppose that either of the following holds.
\begin{enumerate}
\item
$\chark k\neq 2$ and $nh\cF=0$;
\item
$\chark k= 2$ and $n \cF=0$.
\end{enumerate}
Then the restriction to $\{x\}$ gives an isomorphism
\[
i_x^*\colon \cF(\Spec \struct_{X,x}^h) \xrightarrow{\simeq} \cF(\Spec k(x)).
\]
Here $\cF(\Spec \struct_{X,x}^h)=\varinjlim \cF(U_\alpha)$ with the limit taken along all the \etale neighborhoods of $x$ in $X$.
\end{theorem}

Cohomology theories representable in the motivic stable homotopy category are obviously homotopy invariant and $(\PP^1,\T)$-stable whence as a consequence we obtain the following rigidity result.
\begin{theorem}[Theorem~\ref{thm:rigidity_motivic}] \label{thm:rigidity_motivic_intro}
Let $k$ be a field, $X$ be a smooth variety over $k$ and $x\in X$ be a closed point such that $k(x)/k$ is separable. Let $E\in \SH(k)$ and $n\in\NN$ be invertible in $k$. Suppose that either of the following holds.
\begin{enumerate}
\item
$\chark k\neq 2$ and $nh^{\SH}E=0$ for $h^{\SH}=\la 1\ra + \la -1 \ra \in \Hom_{\SH(k)}(\SSp,\SSp)$ (see Definition~\ref{def:hsh});
\item
$\chark k= 2$ and $n E=0$.
\end{enumerate}
Then for $p,q\in \Z$ the restriction to $\{x\}$ gives an isomorphism
\[
i_x^*\colon E^{p,q}(\Spec \struct_{X,x}^h) \xrightarrow{\simeq} E^{p,q}(\Spec k(x)).
\]
Here 
\begin{gather*}
E^{p,q}(\Spec \struct_{X,x}^h)=\varinjlim\limits_\alpha \Hom_{\SH(k)}(\Sigma^\infty_{\PP^1} (U_\alpha)_+, \Sigma^{q}_{\PP^1}E[p-2q]),\\
E^{p,q}(\Spec k(x))=\Hom_{\SH(k)}(\Sigma^\infty_{\PP^1} (\Spec k(x))_+, \Sigma^{q}_{\PP^1}E[p-2q])
\end{gather*}
with the limit taken along the \etale neighborhoods of $x$ in $X$.
\end{theorem}

\noindent
If the base field is perfect then Morel's computation $\Hom_{\SH(k)}(\SSp,\SSp)\cong \GW(k)$ (\cite[Theorem~6.4.1]{Mor04} and \cite[Corollary~6.43]{Mor12}) gives the following reformulation of Theorem~\ref{thm:rigidity_motivic_intro}. The statement was brought to our attention by Tom Bachmann.
\begin{corollary}[Corollary~\ref{cor:perfect_rigidity}]
	Let $k$ be a perfect field, $X$ be a smooth variety over $k$ and $x\in X$ be a closed point. Let $E\in \SH(k)$ and suppose that $\phi E=0$ for some $\phi\in\GW(k)\cong \Hom_{\SH(k)}(\SSp,\SSp)$ such that $\rank \phi$ is invertible in $k$. Then for $p,q\in \Z$ the restriction to $\{x\}$ gives an isomorphism
	\[
	i_x^*\colon E^{p,q}(\Spec \struct_{X,x}^h) \xrightarrow{\simeq} E^{p,q}(\Spec k(x)).
	\]
	Here $E^{p,q}(\Spec \struct_{X,x}^h)=\varinjlim E^{p,q}(U_\alpha)$ with the limit taken along the \etale neighborhoods of $x$ in $X$.
\end{corollary}

It is well known that Theorem~\ref{thm:intr_rigidity} follows from the following one via a geometric argument (see \cite[Proof of Theorem~2]{Gab92}, \cite[Proof of Theorem~4.4]{SV96} or the proof of Theorem~\ref{thm:rigidity} of the present paper). We say that $\relcurve\to S$ admits a fine compactification if there exists a projective closure $\relcurve\subset \relpcurve$ such that $\relpcurve-\relcurve$ is finite over $S$ (see Definition~\ref{def:fine_compactification}).

\begin{theorem}[Corollary~\ref{cor:rigidity_lemma}]
Let $S=\Spec R$ be the spectrum of a Henselian local ring, $\relcurve\to S$ be a smooth morphism of relative dimension $1$ admitting a fine compactification and $r_0,r_1\colon S\to \relcurve$ be morphisms of $S$-schemes such that $r_0(x)=r_1(x)$ for the closed point $x\in S$. Let $\cF$ be a homotopy invariant stable linear framed presheaf over $S$ and $n\in\NN$ be invertible in $R$. Suppose that either of the following holds.
\begin{enumerate}
\item
$2\in R^*$ and $nh\cF=0$;
\item
$2=0$ in $R$ and $n \cF=0$.
\end{enumerate}
Then $r_1^*=r_0^*\colon \cF(\relcurve) \to \cF(S)$.
\end{theorem}

This theorem follows from the following one that deals only with framed correspondences (see Section~\ref{sect:def} for the notation). 

\begin{theorem}[Theorem~\ref{thm:main}] \label{thm:intr_main}
Let $S=\Spec R$ be the spectrum of a Henselian local ring, $\relcurve\to S$ be a smooth morphism of relative dimension $1$ admitting a fine compactification and $r_0,r_1\colon S\to \relcurve$ be morphisms of $S$-schemes such that $r_0(x)=r_1(x)$ for the closed point $x\in S$. Then for every $n\in \NN$ such that $n\in R^*$ the following holds.
\begin{enumerate}
\item
If $2\in R^*$ then
\[
\la \sigma_{\relcurve}^m \ra \circ r_1 - \la \sigma_{\relcurve}^m \ra \circ r_0 = H \circ i_1 - H \circ i_0   + n h_{\relcurve}\circ a
\]
for some $m\in \mathbb{N}$, $H\in \ZF_m^S(\A^1\times S,\relcurve)$ and $a\in \ZF_{m-1}^S(S,\relcurve)$.
\item
If $2=0$ in $R$ then 
\[
\la \sigma_{\relcurve}^m \ra \circ r_1 - \la \sigma_{\relcurve}^m \ra \circ r_0 = H \circ i_1 - H \circ i_0   + n \sigma_\relcurve \circ a
\]
for some $m\in \mathbb{N}$, $H\in \ZF_m^S(\A^1\times S,\relcurve)$ and $a\in \ZF_{m-1}^S(S,\relcurve)$.
\end{enumerate}
Here $i_0,i_1\colon S\to \A^1\times S$ are the closed immersions given by $\{0\}\times S$ and $\{1\}\times S$ respectively.
\end{theorem}
One can regard this theorem as a framed analog of the divisibility properties of Suslin homology $\mathrm{H}_0^{s}(\relcurve/S)$ that arise from its identification with the relative Picard group \cite[Theorem~3.1]{SV96} and divisibility properties of the Picard group. 

Let us give a sketch of the proof of Theorem~\ref{thm:intr_main} assuming that $2\in R^*$. First observe that if we prove the theorem for an open subscheme $\relcurve'\subset \relcurve$ then we get the claim for $\relcurve$ as well, thus we may shrink $\relcurve$ at will. Let $\relcurve\subset \relpcurve$ be a fine compactification of $\relcurve$. Set $\cL_0=\struct_{\relpcurve}(r_0(S))$, $\cL_1=\struct_{\relpcurve}(r_1(S))$. It follows from the rigidity property of the \etale cohomology with finite coefficients (see Lemma~\ref{lm:Picard_rigidity}) that there exists a line bundle $\cL$ over $\relpcurve$ such that $\cL_0\cong \cL_1\otimes \cL^{\otimes 2n}$ and $\cL|_{\overline{C}}\cong \struct_{\overline{C}}$, where $\overline{C}$ is the closed fiber of $\relpcurve$. The line bundles $\cL_0$ and $\cL_1$ are equipped with section $s_0$ and $s_1$ such that the zero loci are given by $Z(s_0)=r_0(S)$, $Z(s_1)=r_1(S)$. Without loss of generality we may assume that $s_0|_{\overline{C}}=s_1|_{\overline{C}}$ (up to the isomorphism $\cL_0|_{\overline{C}}\cong \cL_1|_{\overline{C}}$). Twisting with a sufficiently high power of $\struct_{\relpcurve}(1)$ we choose 
\[
\zeta\in\Gamma(\relpcurve,\cL_0(2nN)),\quad \xi_0\in\Gamma(\relpcurve,\struct_{\relpcurve}(N)),\quad \xi_1\in\Gamma(\relpcurve,\cL(N))
\]
such that $s_0\otimes\xi_0^{\otimes 2n}|_{\overline{C}\cup Z(\zeta)}=s_1\otimes \xi_1^{\otimes 2n}|_{\overline{C}\cup Z(\zeta)}$ (again, up to the isomorphism of line bundles) and 
\[
\relpcurve-\relcurve\subset Z(\zeta),\quad Z(\zeta)\cap Z(\xi_0)=Z(\zeta)\cap Z(\xi_1)=Z(s_0)\cap Z(\xi_0)=Z(s_1)\cap Z(\xi_1)=\emptyset.
\]
Then, identifying $\cL_0(2nN)=\cL_0\otimes \struct_{\relpcurve}(N)^{\otimes 2n}\cong \cL_1\otimes \cL(N)^{\otimes 2n}$, we obtain morphisms
\begin{gather*}
[ts_0\otimes \xi_0^{\otimes 2n}+ (1-t)s_1\otimes \xi_1^{\otimes 2n} : \zeta ] \colon \A^1\times \relpcurve \to \PP^1_S, \\
\frac{\Upsilon_1}{\Upsilon_0}=\frac{ts_0\otimes \xi_0^{\otimes 2n}+ (1-t)s_1\otimes \xi_1^{\otimes 2n}}{\zeta} \colon \A^1\times (\relpcurve-Z(\zeta)) \to \A^1_S,
\end{gather*}
with $t$ being the coordinate on $\A^1$. One can easily see that $Z(\Upsilon_1)=Z(\tfrac{\Upsilon_1}{\Upsilon_0})$ is finite over $\A^1\times S$ whence $\tfrac{\Upsilon_1}{\Upsilon_0}$ gives a homotopy between $\tfrac{s_0\otimes \xi_0^{\otimes 2n}}{\zeta}$ and $\tfrac{s_1\otimes \xi_1^{\otimes 2n}}{\zeta}$. 

At this point we obtained a kind of a ``refined'' proof of the rigidity property for presheaves with $\operatorname{Cor}$-transfers. The divisor of $\tfrac{\Upsilon_1}{\Upsilon_0}\in R[\relpcurve-Z(\zeta)]$ gives an element of $\operatorname{Cor}(\A^1\times S,\relcurve)$ that is a homotopy between the divisors of  $\tfrac{s_0\otimes \xi_0^{\otimes 2n}}{\zeta}$ and $\tfrac{s_1\otimes \xi_1^{\otimes 2n}}{\zeta}$. Expanding 
\[
\divisor_0 \tfrac{s_0\otimes \xi_0^{\otimes 2n}}{\zeta} = [r_0(S)] + 2n \divisor_0 (\xi_0),\quad \divisor_0 \tfrac{s_1\otimes \xi_1^{\otimes 2n}}{\zeta} = [r_1(S)] + 2n \divisor_0 (\xi_1)
\]
one obtains the claim in the $\operatorname{Cor}$-setting.

In order to obtain the claim in the framed setting choose some regular functions $\phi_1,\hdots,\phi_m$ on $U\subset \A^{m+1}_S$ such that $\relcurve=Z(\phi_1,\hdots,\phi_m)$ and choose an \etale neighborhood $W\to U$ of $\relcurve$ with a retraction $\rho\colon W\to \relcurve$ (recall that we could shrink $\relcurve$ from the beginning). Thus we have a framed homotopy
\[
(Z(\Upsilon_1), \A^1\times W,(\phi_1,\hdots,\phi_m,\tfrac{\Upsilon_1}{\Upsilon_0}),\rho)\in \Fr_{m+1}(\A^1\times S,\relcurve)
\]
yielding
\begin{multline*}
(r_0(S)\sqcup Z(\xi_0), W,(\phi_1,\hdots,\phi_m,\tfrac{s_0\otimes \xi_0^{\otimes 2n}}{\zeta}\circ \rho),\rho)\simA \\
\simA (r_1(S)\sqcup Z(\xi_1), W,(\phi_1,\hdots,\phi_m,\tfrac{s_1\otimes \xi_1^{\otimes 2n}}{\zeta}\circ \rho),\rho)\in \Fr_{m+1}(S,\relcurve).
\end{multline*}
Decomposing the framed correspondences with respect to the decomposition of supports (and possibly replacing $\phi_m$ with $\alpha\phi_m$ for some $\alpha \in R^*$ and slightly modifying $\xi_0$) one sees that the claim of the Theorem~\ref{thm:intr_main} follows from the next two lemmas: the first one is a tool to recognize the suspension of a morphism of schemes and the second one is a framed version of the claim that the morphism $\T\to\T$ given by $x\mapsto x^{2n}$ corresponds to $nh\in\GW(k)$ in the motivic stable homotopy category.

\begin{lemma}[Lemma~\ref{lm:suspension}]
Let $S=\Spec R$ be the spectrum of a local ring and $\relcurve$ be a scheme over $S$. Consider an explicit framed correspondence 
\[
(Z,U,\phi,g)=(Z,p\colon U\to \A^m_S,(\phi_1,\phi_2,\hdots,\phi_m),g)\in \Fr_m(S,\relcurve).
\]
Suppose that
\begin{enumerate}
\item
the projection $Z\to S$ is an isomorphism;
\item
$\det \left(\left(\tfrac{\partial \phi_i}{\partial x_j}\right)_{i,j}\right)=1$, where $x_j$-s are the local coordinates around $Z$ given by the \'etaleness of $p$ and the standard coordinates on $\A^m$.
\end{enumerate}
Then $(Z,U,\phi,g)\simA \sigma_\relcurve^m \circ g \circ \rho^{-1}$ with $\rho\colon Z\to S$ being the projection.
\end{lemma}

\begin{lemma}[Lemma~\ref{lm:hyperbolic}]
Let $S$ be a scheme and $X,Y$ be schemes over $S$. Consider an explicit framed correspondence $(Z,U,(\phi_1,\hdots,\phi_{m-1},\alpha\phi_{m}^{2n}),g)\in \Fr_{m}(X,Y)$ such that $\alpha\in\Gamma(U,\struct_U^*)$ and $(\phi_{m}, \pi_X)\colon Z(\phi_1,\hdots,\phi_{m-1})\to \A^1\times X$ is finite with $\pi_X\colon Z(\phi_1,\hdots,\phi_{m-1}) \to X$ being the projection. Then
\[
\sigma_Y\circ \la Z,U,(\phi_1,\hdots,\phi_{m-1},\alpha\phi_{m}^{2n}),g \ra \simA nh_Y \circ \la Z,U,(\phi_1,\hdots,\phi_{m-1},\alpha\phi_{m}),g\ra .
\]
\end{lemma}
Both lemmas are proved via explicit manipulations with framed correspondences.

The paper is organized as follows. In Section~2 we recall the definitions of framed correspondences. In Section~3 we prove a number of technical lemmas constructing framed homotopies. Section~4 deals with the basic properties of framed presheaves. In Section~5 we recall the divisibility properties of the Picard group and prove some technical lemmas of geometric nature about the general sections of very ample line bundles and relative curves. In Section~6 we prove the rigidity theorem for framed presheaves. In Section~7 we give the construction of framed transfers for a representable cohomology theory and derive the corresponding rigidity theorem.

\subsection*{Acknowledgments} The authors would like to thank Ivan Panin and the participants of the seminar on $\A^1$-homotopy and $\K$-theory at St. Petersburg for many helpful discussions. The last part of the work was done during the first author's stay at the Institute Mittag-Leffler whose hospitality he gratefully acknowledges. The research is supported by the Russian Science Foundation grant 14-21-00035

\section{Preliminaries on framed correspondences} \label{sect:def}

All schemes are supposed to be separated and noetherian.

\begin{definition}
Let $X$ be a scheme and $Z$ be a closed subscheme of $X$. An \textit{\etale neighborhood} of $Z$ in $X$ is a pair of morphisms $(p\colon U\to X,\, r\colon Z\to U)$ where $p$ is \etale and $p\circ r=i$ for the closed immersion $i\colon Z\to X$.
\[
\xymatrix{
 & & U \ar[dll]_{\et}^p \\
X & & Y \ar[ll]^{i} \ar[u]_r
}
\]
\end{definition}

\begin{definition}
Let $S$ be a scheme, $X,Y$ be schemes over $S$ and $\relcurve$ be a scheme over $S$ of relative dimension $d$. An \textit{explicit $\relcurve$-inner framed correspondence} consists of the following data:
\begin{enumerate}
\item
a closed subscheme $Z$ of $X\times_S \relcurve$ which is finite over $X$;
\item
an \etale neighborhood $(p\colon U\to X\times_S \relcurve,\, r\colon Z\to U)$ of $Z$ in $X\times_S \relcurve$;
\item
a collection of regular functions $\phi=(\phi_1,\phi_2,\hdots,\phi_d)$ on $U$ such that $r(Z)^{\red}=Z(\phi)^{\red}$ where $Z(\phi)$ stands for the common zero locus of $\phi_i$-s;
\item
a morphism of $S$-schemes $g\colon U\to Y$ with the structure morphism $U\to S$ given by the composition $U\to X\times_S \relcurve\to X\to S$.
\end{enumerate}
\[
\xymatrix{
 & & U \ar[lld]_{\et}^(0.3)p \ar[rr]^\phi \ar[ddrr]^(0.3)g & & \A^d \\
X\times_S \relcurve \ar[drr] & & Z \ar[d]^{\fin} \pour \ar[ll]_(0.3){\cl} \ar[u]^r \ar[rr] & & \{0\} \ar[u] \\
 &  & X \ar[dr] & & Y \ar[dl] \\
 &  & & S &  
}
\]
We usually write an explicit framed correspondence as 
\[
(Z,U,\phi,g)=(Z,(p\colon U\to X\times_S \relcurve,\, r\colon Z\to U),\phi,g).
\]

Two explicit $\relcurve$-inner framed correspondences $(Z,U,\phi,g)$ and $(Z',U',\phi',g')$ are said to be equivalent if $Z^{\red}={Z'}^{\red}$ and there exists an \etale neighborhood $W$ of $Z$ in $U\times_{X\times_S \relcurve} U'$ such that $g\circ \pi_U=g'\circ \pi_{U'}$ and $\phi\circ \pi_U=\phi'\circ \pi_{U'}$ for the respective projections $\pi_U\colon W\to U$ and $\pi_{U'}\colon W\to U'$. The set of $\relcurve$-inner framed correspondences (i.e. explicit $\relcurve$-inner framed correspondences up to the above equivalence) is denoted $\Fr_\relcurve^S(X,Y)$. Set
\[
\ZF^S_\relcurve(X,Y)=\Z[\Fr^S_\relcurve(X,Y)]/A,
\]
where $\Z[\Fr^S_\relcurve(X,Y)]$ is the free abelian group on the set of $\relcurve$-inner framed correspondences and $A$ is the subgroup generated by the elements
\[
\la Z\sqcup Z', U, \phi, g\ra - \la Z, U-Z', \phi|_{U-Z'}, g|_{U-Z'}\ra - \la Z', U-Z, \phi|_{U-Z}, g|_{U-Z}\ra.
\]
For $a\in \Fr^S_\relcurve(X,Y)$ we denote the corresponding element $1\cdot a$ in $\ZF^S_\relcurve(X,Y)$ by $\la a\ra$.

Functoriality of $\Fr_\relcurve^S(X,Y)$ with respect to the morphisms of $S$-schemes gives rise to presheaves $\Fr_\relcurve^S(-,Y)$ and $\ZF_\relcurve^S(-,Y)$ on the category $\Sch_S$. 

An open immersion $q\colon \relcurve'\to \relcurve$ gives rise to a map
\begin{gather*}
\Fr_{\relcurve'}^S(X,Y) \to \Fr_\relcurve^S(X,Y),\\
(Z,p\colon U \to X\times_S \relcurve',\phi,g)\mapsto ((\id_X\times q)(Z),(\id_X\times q)\circ p\colon U \to X\times_S \relcurve,\phi,g).
\end{gather*}
This rule induces morphisms of presheaves
\[
\Fr_{\relcurve'}^S(-,Y)\to \Fr_\relcurve^S(-,Y),\quad \ZF_{\relcurve'}^S(-,Y)\to \ZF_\relcurve^S(-,Y).
\]
\end{definition}

\begin{definition}
Let $S$ be a scheme and $X,Y$ be schemes over $S$. An \textit{(explicit) framed correspondence of level $n$} is an (explicit) $\A^n_S$-inner framed correspondence. We denote
\[
\Fr_n^S(X,Y)=\Fr_{\A^n_S}^S(X,Y).
\]
Note that $\Fr^S_0(X,Y)=\Hom_{\Sch_{S},\bullet}(X_+,Y_+)$ is the set of the morphisms of pointed $S$-schemes. For a morphism $f\in \Hom_{\Sch_S}(X,Y)$ we usually denote by the same letter the corresponding element of $\Fr^S_0(X,Y)$. Denote
\[
\Fr^S_*(X,Y)=\bigvee_{n\ge 0} \Fr^S_n(X,Y),\quad \ZF^S_*(X,Y)=\bigoplus_{n\ge 0} \ZF^S_n(X,Y).
\]

Let $X,Y$ and $V$ be schemes over $S$ and let $\Phi=(Z,U,\phi,g)\in \Fr^S_n(X,Y)$ and $\Psi=(Z',W,\psi,h)\in \Fr^S_m(Y,V)$ be explicit framed correspondences. Then we compose them in the following way (see the details in \cite{GP14}):
\[
\Psi \circ \Phi = (Z\times_Y Z',U\times_Y W, (\phi\circ \pi_{U}, \psi\circ \pi_W), h\circ \pi_W).
\]
One can show that this rule induces associative compositions
\[
\Fr^S_n(X,Y)\times \Fr^S_m(Y,V)\to \Fr^S_{n+m}(X,V),\quad \ZF^S_n(X,Y)\times \ZF^S_m(Y,V)\to \ZF^S_{n+m}(X,V).
\] 
We denote $\Fr_*(S)$ and $\ZF_*(S)$ the categories with the objects being smooth schemes over $S$ and morphisms given by $\Fr^S_*(-,-)$ and $\ZF^S_*(-,-)$ respectively. There is a obvious functor $\Fr_*(S)\to \ZF_*(S)$.
\end{definition}

\begin{remark}
When the base scheme $S$ is clear from the context we will usually omit the superscript $S$ and write 
\begin{gather*}
\Fr_\relcurve(X,Y)=\Fr^S_\relcurve(X,Y),\quad \ZF_\relcurve(X,Y)=\ZF^S_\relcurve(X,Y),\\
\Fr_*(X,Y)=\Fr^S_*(X,Y),\quad \ZF_*(X,Y)=\ZF^S_*(X,Y).
\end{gather*}
\end{remark}

\begin{definition}
For a scheme $Y$ fix the notation for the following framed correspondences.
\begin{align*}
&\sigma_Y=(Y\times \{0\},Y\times \A^1, x,\pi_Y)\in \Fr_1(Y,Y),\\
&h_Y=\la Y\times \{0\},Y\times \A^1, x,\pi_Y\ra +\la Y\times \{0\},Y\times \A^1, -x,\pi_Y\ra\in \ZF_1(Y,Y).
\end{align*}
Here $x$ is the coordinate function on $\A^1$ and $\pi_Y\colon Y\times\A^1\to Y$ is the projection. For $m\ge 1$ the $m$-fold composition of $\sigma_Y$ is denoted $\sigma^m_Y\in \Fr_m(Y,Y)$.
%
\end{definition}

\begin{definition} \label{def:normal_framing}
Let $S$ be a scheme and $\relcurve$ be a scheme over $S$ of relative dimension $d$. A \textit{level $m$ normal framing of $\relcurve$} consists of the following data:
\begin{enumerate}
\item
an open immersion $j\colon W\to \A^{d+m}_S$;
\item
a closed immersion $i\colon \relcurve \to W$;
\item
an \etale neighborhood $(p\colon \widetilde{W}\to W, r\colon \relcurve \to \widetilde{W})$ of $\relcurve$ in $W$;
\item
a collection of regular functions $\psi=(\psi_1,\psi_2,\hdots,\psi_m)$ on $\widetilde{W}$ such that $r(\relcurve)=Z(\psi)$ where $Z(\psi)$ stands for the common zero locus of $\psi_i$-s;
\item
a regular morphism $\rho\colon \widetilde{W}\to \relcurve$ such that $\rho\circ r=\id_\relcurve$.
\end{enumerate}
\[
\xymatrix{
& & & \widetilde{W} \ar[dll]^(0.3)p_\et \ar[rrdd]^(0.3)\rho \ar[rr]^\psi & & \A^{m} \\
\A^{d+m}_S & W \ar[l]^(0.4)j_(0.4){\op} & & \relcurve \ar[drr]_{\id_\relcurve} \ar[ll]^(0.4)i_(0.4){\cl} \ar[u]^r \ar[rr] \pour & & \{0\} \ar[u] \\
& & & & & \relcurve
}
\]
The set of level $m$ normal framings of $\relcurve$ is denoted $\lF_m(\relcurve)$. An open immersion $\relcurve'\subset \relcurve$ induces a map $\lF_m(\relcurve)\to \lF_m(\relcurve')$ given by
\[
(j\colon W\to \A^{d+m}_S,i\colon \relcurve\to W,p\colon \widetilde{W}\to W,\psi,\rho) \mapsto (j'\colon W'\to \A^{d+m}_S,i'\colon \relcurve'\to W',p'\colon \widetilde{W}'\to W',\psi',\rho')
\]
with $W'=W-i(\relcurve-\relcurve'),\,\widetilde{W}'=\widetilde{W}-\rho^{-1}(\relcurve-\relcurve')-p^{-1}(i(\relcurve-\relcurve'))$ and the morphisms being the restrictions of the corresponding morphisms.
\end{definition}

\begin{definition}
Let $S$ be a scheme, $X$ and $Y$ be schemes over $S$ and $\relcurve$ be a scheme over $S$ of relative dimension $d$. For a level $m$ normal framing of $\relcurve$ and an explicit $\relcurve$-inner framed correspondence from $X$ to $Y$ set
\begin{multline*}
(j\colon W\to \A^{d+m}_S,i\colon \relcurve\to W,q\colon \widetilde{W}\to W,\psi,\rho) \frp (Z,p\colon U\to X\times_S \relcurve,\phi,g)= \\
=((\id_X\times (j\circ i)) (Z), (\pi_X \circ p\circ \pi_U, j\circ q\circ \pi_{\widetilde{W}})\colon U\times_\relcurve \widetilde{W} \to X\times_S \A^{d+m}_S,(\phi\circ \pi_U, \psi\circ \pi_{\widetilde{W}}),g\circ \pi_U).
\end{multline*}
\[
\xymatrix @C=3pc {
& & U \times_\relcurve \widetilde{W} \podr \ar[r]_(0.6){\pi_U} \ar[d] \ar[ddll]_(0.4){\et} \ar@/^1.0pc/[rrr]_(0.7){\phi\times\psi}  & U \ar[d]_p \ar@/^1.0pc/[dddrr]^g & & \A^{d+m}\\
& & X\times_S \widetilde{W} \ar[r]^{\id \times \rho} \ar[dl]_{\et}^(0.4){\id \times q} &  X \times_S  \relcurve & & \\
X \times_S \A^{d+m}_S \ar@/_1.0pc/[drrr] & X \times_S W \ar[l]^(0.45){\id\times j}_(0.45){\op}  &  X \times_S \relcurve \ar[dr] \ar[l]_(0.4){\cl}^(0.4){\id \times i} \ar[ur]^{\id} \ar[u] & & Z\ar[ll]_{\cl}\ar[ul]^{\cl} \ar[uul]\ar[r]\ar[dl]^{\fin} & \{0\} \ar[uu] \\
& & & X \ar[dr] & & Y\ar[dl] \\
& & & & S &  \\   
}
\]

Here the fibered product $ U \times_{\relcurve} \widetilde{W}\cong  U \times_{X\times_S \relcurve} (X\times_S \widetilde{W})$ is taken with respect to $p$ and $\rho$, morphisms $\pi_{\widetilde{W}}\colon U \times_{\relcurve} \widetilde{W}\to \widetilde{W}$, $\pi_U\colon U \times_{\relcurve} \widetilde{W}\to U$ and $\pi_X\colon X \times_S \relcurve\to X$ are the projections.

This rule gives rise to the pairings
\[
\frp\colon \lF_m(\relcurve)\times \Fr_\relcurve(X,Y) \to \Fr_{d+m}(X,Y),\quad \frp\colon \lF_m(\relcurve)\times \ZF_\relcurve(X,Y) \to \ZF_{d+m}(X,Y)
\]
inducing morphisms of sheaves
\[
\frp\colon \lF_m(\relcurve)\times \Fr_\relcurve(-,Y) \to \Fr_{d+m}(-,Y),\quad \frp\colon \lF_m(\relcurve)\times \ZF_\relcurve(-,Y) \to \ZF_{d+m}(-,Y).
\]
\end{definition}

\section{Framed homotopies}
Most of the results of this section are not original and have already appeared in the literature (see \cite{AGP16,GNP16,GP14,GP15}) in slightly different incarnations. For the sake of completeness we give the proofs for the statements in the precise forms that we are going to use.

\begin{definition}
Let $S$ be a scheme, $X,Y$ be schemes over $S$ and $\relcurve$ be a scheme over $S$ of relative dimension $d$. We say that $a,b\in \Fr_\relcurve(X,Y)$ (resp. $\ZF_\relcurve(X,Y)$) are \textit{$\A^1$-homotopic} and denote it $a \simA b$ if there exists a sequence of elements $H_1,H_2,\hdots H_n\in \Fr_\relcurve(\A^1 \times X,Y)$ (resp. $\ZF_\relcurve(\A^1 \times X,Y)$) such that 
\begin{enumerate}
\item
$H_1 \circ i_0=a,\, H_n \circ i_1=b$;
\item
$H_l \circ i_0=H_{l-1} \circ i_1,\, l\ge 2$.
\end{enumerate}
Here $i_0,i_1\colon X\to \A^1\times X$ are the closed immersions given by $\{0\}\times X$ and $\{1\}\times X$ respectively.
\end{definition}

\begin{remark}
For $a,b\in \ZF_\relcurve(X,Y)$ one has $a\simA b$ iff there exists $H\in \ZF_\relcurve(\A^1\times X,Y)$ such that $H\circ i_0=a$ and $H\circ i_1=b$: take 
\[
H=H_1+H_2+\hdots + H_n - c_1 -c_2 - \hdots c_{n-1}
\]
where $c_i= H_{i}\circ i_1 \circ p$ are the constant homotopies with $p\colon \A^1\times X \to X$ being the projection.
\end{remark}

\begin{definition}
Let $S$ be a scheme, $X,Y$ be schemes over $S$ and $\relcurve$ be a scheme over $S$ of relative dimension $d$. For an explicit framed correspondence $(Z,U,\phi,g)\in\Fr_\relcurve(X,Y)$ and $A\in \GL_d(\Gamma(U,\struct_U))$ set
\[
(Z,U,\phi,g) \cdot A = (Z,U,\phi \cdot A,g),
\]
where $\cdot$ on the right stands for the matrix multiplication.

$A\in \GL_d(\Gamma(U,\struct_U))$ is called \textit{elementary} if $A$ is a product of elementary transvections,
\[
A=T_{i_1j_1}(a_1)T_{i_2j_2}(a_2)\hdots T_{i_mj_m}(a_m),
\]
with $a_l\in \Gamma(U,\struct_U)$ for all $l$. Here $i_l\neq j_l$ and $T_{i_lj_l}(a_l)$ differs from the unit matrix only at the position $(i_l,j_l)$ where stands $a_l$. Recall that every matrix of determinant $1$ over a field, a local ring, or over $\Z$ is elementary.
\end{definition}

\begin{lemma}\label{lm:basic_homotopies}
Let $S$ be a scheme, $X,Y$ be schemes over $S$ and $\relcurve$ be a scheme over $S$ of relative dimension $d$. Consider an explicit framed correspondence 
\[
(Z,U,\phi,g)=(Z,U,(\phi_1,\hdots,\phi_d),g)\in \Fr_\relcurve(X,Y).
\]
Then
\begin{enumerate}
\item  \label{lm:basic_homotopies:elementary}
$(Z,U,\phi,g) \simA (Z,U,\phi,g)\cdot A$ for elementary $A\in \GL_d(\Gamma(U,\struct_U))$;

\item \label{lm:basic_homotopies:derivative}
$(Z,U,\phi,g) \simA (Z,U,(\phi_1,\hdots,\phi_{d-1},\alpha\phi_d),g)$ for $\alpha\in \Gamma(U,\struct_U^{\ast})$ satisfying $\alpha|_Z=1$.
\end{enumerate}
\end{lemma}
\begin{proof}
Throughout the proof we denote $t$ the parameter of the considered homotopy (i.e. the coordinate on $\A^1$) and $\pi_U$ the projections onto $U$.

(\ref{lm:basic_homotopies:elementary}) It is sufficient to consider the case of an elementary transvection $A=T_{ij}(a)$. The homotopy is given by
\[
H= (\A^1\times Z,\A^1 \times U,  \phi \cdot T_{ij}(ta),g\circ\pi_U).
\]

(\ref{lm:basic_homotopies:derivative}) Consider a regular function $t\alpha+(1-t)$ on $\A^1 \times U$. It follows from the assumption that $Z(t\alpha+(1-t))\cap (\A^1 \times Z)=\emptyset$. The homotopy is given by
\[
H= (\A^1 \times Z,(\A^1 \times U)-Z(t\alpha+(1-t)),  (\phi_1,\hdots, \phi_{d-1}, (t\alpha+(1-t))\phi_d),g\circ\pi_U).
\]
\end{proof}

\begin{lemma}\label{lm:basic_homotopies2}
Let $S$ be a scheme, $X,Y$ be schemes over $S$ and $d\in \mathbb{N}$. Consider an explicit framed correspondence 
\[
(Z,U,\phi,g)=(Z,U,(\phi_1,\hdots,\phi_d),g)\in \Fr_d(X,Y).
\]
Then
\begin{enumerate}
\item \label{lm:basic_homotopies2:shuffle}
$\sigma_Y^2 \circ (Z,U,\phi,g) \simA (Z,U,\phi,g) \circ \sigma_X^2$ and $h_Y \circ \la Z,U,\phi,g \ra \simA \la Z,U,\phi,g \ra \circ h_X$;
	
\item \label{lm:basic_homotopies2:hyper}
$\la\sigma_Y\ra \circ \left(\la Z,U,(\phi_1,\hdots,\phi_{d-1},\phi_d),g \ra + \la Z,U,(\phi_1,\hdots,\phi_{d-1}, -\phi_{d}),g\ra \right) \simA  h_Y\circ \la Z,U,\phi,g\ra$.
\end{enumerate}
\end{lemma}
\begin{proof}
(\ref{lm:basic_homotopies2:shuffle}) Immediately follows from Lemma~\ref{lm:basic_homotopies}(\ref{lm:basic_homotopies:elementary}) applying permutation matrices.

(\ref{lm:basic_homotopies2:hyper}) We have
\begin{gather*}
\la\sigma_Y\ra \circ \la Z,U,(\phi_1,\hdots,\phi_{d-1},\phi_d),g \ra = \la Z\times\{0\},U\times \A^1,(\phi_1,\hdots,\phi_{d-1},\phi_d, x),g \ra,\\
\la\sigma_Y\ra \circ \la Z,U,(\phi_1,\hdots,\phi_{d-1},-\phi_d),g \ra = \la Z\times\{0\},U\times \A^1,(\phi_1,\hdots,\phi_{d-1},-\phi_d, x),g \ra.
\end{gather*}
Applying Lemma~\ref{lm:basic_homotopies}(\ref{lm:basic_homotopies:elementary}) to the elementary matrix $\mathrm{diag}(1,1,\hdots,1,-1,-1)$ we obtain
\[
\la Z\times\{0\},U\times \A^1,(\phi_1,\hdots,\phi_{d-1},-\phi_d, x),g \ra \simA \la Z\times\{0\},U\times \A^1,(\phi_1,\hdots,\phi_{d-1},\phi_d, -x),g \ra,
\]
whence the claim.	
\end{proof}

\begin{lemma} \label{lm:hyperbolic}
Let $S$ be a scheme, $X,Y$ be schemes over $S$ and $\relcurve$ be a scheme over $S$ of relative dimension $1$. Let $(Z,U,\phi_1,g)\in \Fr_\relcurve(X,Y)$ be an explicit framed correspondence such that $(\phi_1, \pi_X)\colon U\to \A^1\times X$ is finite for the composition $\pi_X\colon U\to X\times_S \relcurve\to X$. Then for every $\alpha\in\Gamma(U,\struct_U^*)$ one has
\begin{enumerate}
\item
$\la Z,U,\alpha\phi_1^{2n},g \ra \simA n\left(\la Z,U,\alpha\phi_1,g\ra + \la Z,U,-\alpha\phi_1,g \ra \right) \in \ZF_\relcurve(X,Y)$;
\item
$\la Z,U,\alpha\phi_1^{2n+1},g \ra \simA \la Z,U,\alpha\phi_1,g\ra + n\left(\la Z,U,\alpha\phi_1,g\ra + \la Z,U,-\alpha\phi_1,g \ra \right) \in \ZF_\relcurve(X,Y)$.
\end{enumerate}
Moreover, for every normal framing $\Phi\in \lF_{m-2}(\relcurve)$ one has
\begin{enumerate}
\item
$\sigma_Y \circ (\Phi \frp \la Z,U,\alpha\phi_1^{2n},g \ra) \simA n h_Y \circ (\Phi\frp \la Z,U,\alpha\phi_1,g \ra) \in \ZF_{m}(X,Y)$;
\item
$\sigma_Y \circ (\Phi \frp \la Z,U,\alpha\phi_1^{2n+1},g \ra) \simA \sigma_Y \circ (\Phi\frp \la Z,U,\alpha\phi_1,g \ra) + n h_Y \circ (\Phi\frp \la Z,U,\alpha\phi_1,g \ra) \in \ZF_{m}(X,Y)$.
\end{enumerate}
\end{lemma}
\begin{proof}
Let $t$ be the coordinate on $\A^1$ and consider the regular function $(\phi_1+t)\phi_1^N\in\Gamma(\A^1\times U,\struct_{\A^1\times U})$. The zero locus decomposes as 
\[
Z((\phi_1+t)\phi_1^N)^{\red}=(\A^1\times Z(\phi_1))^{\red}\cup \Gamma^T_{-\phi_1},
\]
where $\Gamma^T_{-\phi_1}$ is the transpose of the graph of $-\phi_1$. The zero set $Z^\red=Z(\phi_1)^{\red}$ is finite over $X$ by the assumption, thus $(\A^1\times Z(\phi_1))^{\red}$ is finite over $\A^1\times X$. The graph $\Gamma^T_{-\phi_1}$ is isomorphic to $U$ and the projection $\Gamma^T_{-\phi_1}\to \A^1\times X$ is given by $(\phi_1, \pi_X)$ which is finite by the assumption of the lemma. Then $Z((\phi_1+t)\phi_1^N)$ is also finite over $\A^1\times X$. Thus we have an explicit framed correspondence 
\[
( Z((\phi_1+t)\phi^{N}_1), \A^1\times U, \alpha(\phi_1+t)\phi^{N}_1, g\circ \pi_U) \in \Fr_{\relcurve}(\A^1\times X,Y)
\]
with $\pi_U\colon \A^1\times U\to U$ being the projection. Hence
\begin{multline*}
\la Z,U,\alpha\phi_1^{N+1},g \ra \simA \la Z (\phi_1 ),U-Z(\phi_1 + 1),\alpha(\phi_1 + 1 )\phi_1^{N}|_{U-Z(\phi_1+1)},g|_{U-Z(\phi_1+1)} \ra + \\ + \la Z(\phi_1+ 1),U-Z,\alpha(\phi_1+ 1) \phi_1^{N}|_{U-Z},g|_{U-Z} \ra.
\end{multline*}
We have $(\phi_1+1)|_{Z(\phi_1)}=1$ and $\phi_1|_{Z(\phi_1+1)}=-1$, thus Lemma~\ref{lm:basic_homotopies}(\ref{lm:basic_homotopies:derivative}) yields
\begin{multline*}
\la Z(\phi_1) ,U-Z(\phi_1 + 1),\alpha(\phi_1 + 1 )\phi_1^{N}|_{U-Z(\phi_1+1)},g|_{U-Z(\phi_1+1)} \ra \simA\\ 
\simA \la Z(\phi_1) ,U-Z(\phi_1 + 1),\alpha\phi_1^{N}|_{U-Z(\phi_1+1)},g|_{U-Z(\phi_1+1)} \ra = \la Z ,U,\alpha\phi_1^{N},g \ra,
\end{multline*}
\begin{multline*}
\la Z(\phi_1+ 1),U-Z,\alpha(\phi_1+ 1) \phi_1^{N}|_{U-Z},g|_{U-Z} \ra \simA\\ 
\simA \la Z(\phi_1+ 1),U-Z,(-1)^N\alpha(\phi_1+ 1)|_{U-Z},g|_{U-Z} \ra = \la Z ,U,(-1)^N\alpha(\phi_1+1),g \ra.
\end{multline*}
The homotopy $\la Z(\phi_1+t),U\times \A^1,(-1)^N\alpha(\phi_1+t),g\circ\pi_U \ra$ yields
\[
\la Z(\phi_1) ,U,(-1)^N\alpha(\phi_1+1),g \ra \simA \la Z(\phi_1) ,U,(-1)^N\alpha\phi_1,g \ra.
\]

Summing up the above we see that
\[
\la Z,U,\alpha\phi_1^{N+1},g \ra \simA \la Z,U,\alpha\phi_1^{N},g \ra + \la Z,U,(-1)^N\alpha\phi_1,g \ra.
\]
Iterating we obtain the first claim of the lemma,
\begin{gather*}
\la Z,U,\alpha\phi_1^{2n},g \ra \simA n\left(\la Z,U,\alpha\phi_1,g\ra + \la Z,U,-\alpha\phi_1,g \ra \right),\\
\la Z,U,\alpha\phi_1^{2n+1},g \ra \simA \la Z,U,\alpha\phi_1,g\ra + n\left(\la Z,U,\alpha\phi_1,g\ra + \la Z,U,-\alpha\phi_1,g \ra \right).
\end{gather*}

The second claim of the lemma follows from the above equivalences and Lemma~\ref{lm:basic_homotopies2}(\ref{lm:basic_homotopies2:hyper}).
\end{proof}

\begin{lemma} \label{lm:framed_function}
Let $S$ be a scheme, $X,Y$ be schemes over $S$ and 
\[
(X\times\{0\}, (p\colon U\to X\times \A^m, r\colon X\times\{0\}\to U), \phi, g)\in \Fr_m(X,Y)
\]
be an explicit framed correspondence. Then
\[
(X\times\{0\},U,\phi,g) \simA (X\times\{0\},U,\phi,g')
\]
for any morphism of $S$-schemes $g'\colon U\to Y$ satisfying $g\circ r=g'\circ r$.
\end{lemma}

\begin{proof}
We have $p^{-1}(X\times \{0\})=r(X\times \{0\})\sqcup \widetilde{X}$. Shrinking $U$ to $U-\widetilde{X}$ we may assume that $p$ is an isomorphism over $X\times\{0\}$.

Consider the morphism
\[
F\colon \A^1\times X\times \A^m \to X\times \A^m, \quad (t,u,x_1,\hdots,x_m)\mapsto (u,tx_1,\hdots,tx_m)
\]
and the Cartesian squares
\[
\xymatrix @C=4pc {
U_F \podr \ar[r]^{F_2} \ar[d]_{q} & U'_F \ar[r]^{F_1}\podr \ar[d] & U \ar[d]^p \\
\A^1\times U \ar[r]^(0.45){\id_{\A^1}\times p} & \A^1\times X\times \A^m \ar[r]^(0.55)F & X\times \A^m
}
\]
Let $\pi_U\colon \A^1\times U\to U$ be the projection. A straightforward computation shows that the homotopy
\[
(\A^1\times X\times \{0\}, (\id_{\A^1}\times p) \circ q\colon U_F\to \A^1\times X\times \A^m, \phi\circ \pi_U \circ q, g\circ F_1\circ F_2)\in \Fr_m(\A^1\times X,Y)
\]
provides an equivalence
\[
(X\times\{0\},U,\phi,g) \simA (X\times\{0\},U,\phi,g\circ r\circ \pi_X)
\]
with $\pi_X\colon U\to X$ being the projection. The same argument yields
\[
(X\times\{0\},U,\phi,g') \simA (X\times\{0\},U,\phi,g'\circ r \circ \pi_X)
\]
and the claim follows.
\end{proof}

\begin{definition} \label{def:Jac}
Let $S=\Spec R$ be the spectrum of a local ring and $Y$ be a scheme over $S$. Consider an explicit framed correspondence 
\[
(Z,U,\phi,g)=(Z,(p\colon U\to \A^m_S,r\colon Z\to U),(\phi_1,\phi_2,\hdots,\phi_m),g)\in \Fr_m^S(S,Y)
\]
and suppose that the projection $Z\to S$ is an isomorphism. Let $y\in Z$ be the closed point and $p^*\colon \struct_{\A^m_S,y} \to \struct_{U,r(y)}$ be the morphism of local rings induced by $p$. Denote $x_1,x_2,\hdots,x_m$ the standard coordinates on $\A^m_S$. Then $Z$ is defined in $\Spec \struct_{\A^m_S,y}$ (as well as in $\A^m_S$) by the ideal $I=(x_1-a_1,x_2-a_2,\hdots,x_m-a_m)$ for some $a_1,a_2,\hdots,a_m\in R$. Since $p\colon U\to \A^m_S$ is an \etale neighborhood of $Z$ then $p^*$ induces an isomorphism $I/I^2\cong p^*(I)/p^*(I)^2$ of free $R$-modules with the canonical basis on the left given by the classes $\overline{x}_1-a_1,\overline{x}_2-a_2,\hdots,\overline{x}_m-a_m$. Let $J\in \mathrm{M}_m(R)$ be the matrix that takes $(\overline{x}_i-a_i)$-s to $\overline{\phi}_j$-s,
\[
(\overline{x}_1-a_1,\overline{x}_2-a_2,\hdots,\overline{x}_m-a_m)\cdot J = (\overline{\phi}_1,\overline{\phi}_2,\hdots, \overline{\phi}_m).
\]
We denote
\[
\Jac(Z,U,\phi,g)=\det J \in R
\]
and refer to it as \textit{the Jacobian of $(Z,U,\phi,g)$}.
\end{definition}

\begin{lemma} \label{lm:suspension}
Let $S=\Spec R$ be the spectrum of a local ring and $Y$ be a scheme over $S$. Consider an explicit framed correspondence 
\[
(Z,U,\phi,g)=(i\colon Z\to \A^m_S,(p\colon U\to \A^m_S,r\colon Z\to U),(\phi_1,\phi_2,\hdots,\phi_m),g)\in \Fr_m^S(S,Y).
\]
Suppose that
\begin{enumerate}
\item
the projection $Z\to S$ is an isomorphism;
\item
$\Jac(Z,U,\phi,g)=1$.
\end{enumerate}
Then $(Z,U,\phi,g)\simA \sigma_Y^m \circ g\circ \rho^{-1}$ with $\rho\colon Z\to S$ being the projection.
\end{lemma}

\begin{proof}
The closed immersion $i\circ \rho^{-1}\colon S\to \A^m_S$ is given by 
\[
x_1-a_1, x_2-a_2, \hdots, x_m-a_m \in R[x_1,x_2,\hdots,x_m]
\]
with $a_l\in R,\, 1\le l\le m,$ and $x_l$-s being the standard coordinates on $\A^m$. Consider the morphism
\[
F\colon \A^1\times \A^m_S\to \A^m_S,\quad (t,x_1,x_2,\hdots,x_m)\mapsto (x_1-ta_1,x_2-ta_2,\hdots,x_m-ta_m)
\]
and set
\[
p_F=(\pi_{\A^1}, F\circ (\id_{\A^1}\times p))\colon \A^1\times U\to \A^1\times \A^m_S
\]
with $\pi_{\A^1}\colon \A^1\times U\to \A^1$ being the projection. Consider the homotopy 
\[
(F\circ (\id_{\A^1}\times i) \colon\A^1\times Z \to \A^1\times \A^m_S ,p_F \colon \A^1\times U\to \A^1\times\A^m_S,\phi\circ \pi_U,g\circ \pi_U).
\]
Here $\pi_U\colon \A^1\times U\to U$ is the projection. The homotopy gives rise to an equivalence
\[
(i\colon Z\to \A^m_S,p\colon U\to \A^m_S,\phi,g) \simA (S\times \{0\},F|_{\{1\}\times \A^m_S}\circ p\colon U\to \A^m_S,\phi,g).
\]
In view of the above equivalence and Lemma~\ref{lm:framed_function} from now on we assume that $Z=S\times\{0\}$ and $g=f\circ \pi_S\circ p$ for a morphism $f\in \Hom_{\Sch_S}(S,Y)$ and the projection $\pi_S\colon \A^m_S\to S$, i.e. that
\[
(Z,U,\phi,g) = (S\times\{0\}, (p\colon U\to \A^m_S, r\colon S\to U), \phi, f\circ \pi_S\circ p).
\]

Let $y\in S$ be the closed point. As before, denote $x_1,x_2,\hdots,x_m$ the standard coordinates on $\A^m_S$ and consider $\widetilde{x}_l=p^*x_l\in \Gamma(U,\struct_U), 1\le l\le m,$ the regular functions on $U$ given by $x_l$-s composed with $p$. The first assumption of the lemma yields that the ideal in $\struct_{U,r(y)}$ generated by $\phi_1,\phi_2,\hdots,\phi_m$ is contained in the ideal $\widetilde{I}=(\widetilde{x}_1,\widetilde{x}_2,\hdots,\widetilde{x}_m)$. The second assumption yields that these ideals coincide modulo $\widetilde{I}^2$ whence, by Nakayama's lemma, the ideals coincide. Hence there exists $A\in M_m(\struct_{U,r(y)})$ such that
\[
(\phi_1,\phi_2,\hdots,\phi_m) \cdot A = (\widetilde{x}_1,\widetilde{x}_2,\hdots,\widetilde{x}_m).
\]
By Definition~\ref{def:Jac} we have $\det A =\Jac (Z,U,\phi,g)^{-1} \mmod \widetilde{I}$. Since $\Jac (Z,U,\phi,g)=1$ the matrix $A$ is invertible. Let $U'$ be a Zariski neighborhood of $r(y)$ such that $A$ is defined and invertible over $U'$.

Lemma~\ref{lm:basic_homotopies}(\ref{lm:basic_homotopies:derivative}) yields
\[
(Z,U,(\phi_1,\phi_2,\hdots,\phi_m),g) \simA (Z,U,(\phi_1,\phi_2,\hdots,(\det A)\phi_m),g)
\]
thus we may assume that $\det A=1$. Since $A|_{\Spec \struct_{U',r(y)}}$ is elementary as a matrix of determinant $1$ over a local ring we can choose a Zariski neighborhood $U''$ of $r(y)$ such that $A|_{U''}$ is elementary. Then Lemma~\ref{lm:basic_homotopies}(\ref{lm:basic_homotopies:elementary}) yields
\begin{multline*}
(Z,U,\phi,g) = (S\times\{0\}, p\colon U\to \A^m_S, \phi, f\circ \pi_S\circ p) \simA\\
\simA(S\times\{0\}, p|_{U''}\colon U''\to \A^m_S, (\widetilde{x}_1|_{U''},\widetilde{x}_2|_{U''},\hdots,\widetilde{x}_m|_{U''}), f\circ \pi_S\circ p|_{U''}) = \\
= (S\times\{0\}, \A^m_S, (t_1,t_2,\hdots,t_m), f\circ \pi_S) = \sigma_Y^m\circ f.
\end{multline*}
Here the second to the last equality is given by shrinking the \etale neighborhood of $S\times\{0\}$ in $\A^m_S$ from $\A^m_S$ to $U''$ via $p|_{U''}$. The claim follows.
\end{proof}

\section{Framed presheaves} \label{sect:frpr}

If not otherwise specified, all the presheaves considered below are presheaves of abelian groups.

\begin{definition}
Let $S$ be a scheme. A \textit{framed presheaf} over $S$ is a presheaf on $\Fr_*(S)$. A \textit{linear framed presheaf over $S$} is an additive presheaf on $\ZF_*(S)$. We adopt the following terminology.
\begin{enumerate}
\item
A (linear) framed presheaf $\cF$ is \textit{stable} if
\[
\sigma_Y^*\colon \cF(Y)\to \cF(Y)
\]
is an isomorphism for every smooth scheme $Y$ over $S$.
\item
A (linear) framed presheaf $\cF$ is \textit{homotopy invariant} if
\[
\pi_{Y}^*\colon \cF(Y)\to \cF(\A^1\times Y)
\]
is an isomorphism for every smooth scheme $Y$ over $S$. Here $\pi_Y\colon \A^1\times Y\to Y$ is the projection.
\item
A linear framed presheaf $\cF$ is \textit{$h$-torsion} if
\[
h_Y^*=0\colon \cF(Y)\to \cF(Y)
\]
for every smooth scheme $Y$ over $S$. We shorten the notation as $h\cF=0$.
\end{enumerate}
\end{definition}

\begin{remark}
It is clear that for a homotopy invariant stable (linear) framed presheaf $\cF$ one has $\Phi^*=\Psi^*\colon \cF(Y)\to\cF(X)$ for $\Phi\in\Fr_n(X,Y)$ and $\Psi\in\Fr_{n+m}(X,Y)$ such that $\sigma_Y^m\circ \Phi\simA \Psi$.
\end{remark}

\begin{lemma}[{cf. \cite[the discussion above Theorem~1.1]{GP15}}] \label{lm:additivity}
Let $S$ be a scheme and $\cF$ be a homotopy invariant stable framed presheaf of abelian groups over $S$ such that for every smooth scheme $X$ over $S$ the canonical embeddings $i_1,i_2\colon X\to X\sqcup X$ induce an isomorphism
\[
(i_1^*,i_2^*)\colon \cF(X\sqcup X)\xrightarrow{\simeq} \cF(X)\oplus \cF(X).
\]
Then $\cF$ admits a canonical structure of a homotopy invariant stable linear framed presheaf over $S$, i.e. there exists a canonical linear framed presheaf $\widetilde{\cF}$ fitting in the following commutative diagram.
\[
\xymatrix{
\Fr_*(S) \ar[d] \ar[r]^{\cF} & \mathrm{Ab} \\
\ZF_*(S) \ar[ru]_{\widetilde{\cF}}& 
}
\]
\end{lemma}
\begin{proof}
Let $\ZFr_*(S)$ be the category with the objects being those of $\Fr_*(S)$ and the morphisms given by the free abelian groups $\mathbb{Z}[\Fr_*(X,Y)]$. There exists a natural extension of $\cF$ to $\ZFr_*(S)$, i.e. there is a canonical functor $\cF'$ making the following diagram commute.
\[
\xymatrix{
\Fr_*(S) \ar[d] \ar[r]^{\cF} & \mathrm{Ab} \\
\ZFr_*(S) \ar[ru]_{\cF'} &
}
\]
We need to show that for schemes $X,Y$ smooth over $S$ and an explicit correspondence $\Phi=(Z_1\sqcup Z_2, U,\phi,g)\in \Fr_m(X,Y)$ we have $\Phi^*=\Phi_1^*+\Phi_2^*$ where
\[
\Phi_1 = (Z_1, U-Z_1, \phi,g),\quad \Phi_2=(Z_2, U-Z_2, \phi,g).
\]
Here we shorten the notation omitting the restriction to the respective neighborhoods.

Consider the explicit correspondences 
\[
j_1=(X\sqcup \emptyset,X\sqcup \emptyset, \id_X),\, j_2=(\emptyset\sqcup X,\emptyset\sqcup X, \id_X)\in \Fr_0(X\sqcup X,X)=\Hom_{\Sch_{S},\bullet}((X\sqcup X)_+,X_+).
\]
Here the supports are given by the components of $X\sqcup X$, the neighborhoods coincide with the supports and there are no functions $\phi$ since these are level $0$ correspondences. One clearly has
\[
j_1\circ i_1 = j_2\circ i_2 =\id_X
\]
and since $(i_1^*,i_2^*)\colon \cF(X\sqcup X)\xrightarrow{\simeq} \cF(X)\oplus \cF(X)$ is an isomorphism then $( j_1^*, j_2^*)=(i_1^*,i_2^*)^{-1}$.

Set
\[
\omega=(X\times\{0\}\sqcup X\times \{1\}, X\times(\A^1-\{1\})\sqcup X\times(\A^1-\{0\}) , t\sqcup (t-1), \pi_X\sqcup \pi_X)\in \Fr_1(X,X\sqcup X).
\]
Here the \etale neighborhood is given by the disjoint union $X\times(\A^1-\{1\})\sqcup X\times(\A^1-\{0\})$ with the projection to $X\times \A^1$, the regular function on the components of the \etale neighborhood is given by $t$ and $t-1$ respectively, and the morphism $\pi_X\sqcup \pi_X\colon X\times(\A^1-\{1\})\sqcup X\times(\A^1-\{0\})\to X\sqcup X$ is the component-wise projection.  It is straightforward to check that 
\[
j_1\circ \omega \simA \sigma_X \simA j_2\circ \omega.
\]
Thus the following diagram commutes.
\[
\xymatrix @C=6pc{
\cF(X\sqcup X) \ar[r]^{\omega^*} \ar[d]_{(i_1^*,i_2^*)} & \cF(X) \\
\cF(X)\oplus \cF(X) \ar@/_1.0pc/[ur]^{+}  \ar@/_1.0pc/[u]_{(j_1^*,j_2^*)}& 
}
\]

Denote
\[
\widetilde{\Phi}= (Z_1\sqcup Z_2, U\sqcup U, \phi\sqcup \phi ,g\sqcup g) \in \Fr_m(X\sqcup X,Y),
\]
with $Z_1$ and $Z_2$ placed over the different copies of $X$. It is straightforward to see that $\widetilde{\Phi}\circ \omega \simA \sigma_Y\circ \Phi$ and $\widetilde{\Phi}\circ i_1 = \Phi_1,\, \widetilde{\Phi}\circ i_2 = \Phi_2$. Thus
\[
\Phi^*= \omega^*\circ \widetilde{\Phi}^*= (i_1^*+i_2^*)\circ \widetilde{\Phi}^*= \Phi_1^*+\Phi_2^*.\qedhere
\]
\end{proof}

\begin{definition} \label{def:extend}
Let $k$ be a field. An \textit{essentially smooth scheme} over $k$ is a noetherian $k$-scheme which is the inverse limit of a filtering system $\{Y_\alpha\}$ with each transition morphism $Y_\beta\to Y_\alpha$ being an \etale affine morphism between smooth $k$-schemes. We refer to \cite{GD67} for the properties of the essentially smooth schemes that we use below. The category of essentially smooth $k$-schemes is denoted $\mathrm{EssSm}_k$. A particular example of an essentially smooth $k$-scheme is $\Spec \struct_{X,x}^h$ for a smooth scheme $X$ over $k$ and a point $x\in X$.

Let $\cF$ be a presheaf on $\Sm_k$. There is a canonical extension $\widetilde{\cF}$ of $\cF$ to the category $\mathrm{EssSm}_k$ with $\widetilde{\cF}(Y)=\varinjlim \cF(Y_\alpha)$ for an essentially smooth $k$-scheme $Y=\varprojlim Y_\alpha$.  Let $S$ be an essentially smooth $k$-scheme. Every scheme $Y$ smooth over $S$ can be viewed as an essentially smooth $k$-scheme thus the presheaf $\cF$ on $\Sm_k$ gives rise to a presheaf $\widetilde{\cF}$ on $\Sm_S$.
\end{definition}

\begin{lemma} \label{lm:extension}
Let $k$ be a field and $S$ be an essentially smooth $k$-scheme. Then for every (linear) framed presheaf $\cF$ over $k$ there exists a canonical structure of a (linear) framed presheaf on the associated presheaf $\widetilde{\cF}$ on $\Sm_S$. If $\cF$ is homotopy invariant, stable or $h$-torsion then so is $\widetilde{\cF}$.
\end{lemma}
\begin{proof}
Let $X,Y$ be smooth schemes over $S$. Then $S=\varprojlim S_\alpha, X=\varprojlim X_\alpha$ and $Y=\varprojlim Y_\alpha$ with all the transition morphisms being \etale and affine, $S_\alpha$ being smooth over $k$, $X_\alpha$ and $Y_\alpha$ being smooth over $S_\alpha$ and $X_\beta=S_\beta\times_{S_{\alpha}}X_\alpha$, $Y_\beta=S_\beta\times_{S_{\alpha}}Y_\alpha$ for the structure morphisms $S_\beta\to S_\alpha$. The base change gives rise to morphisms
\[
\Fr^{S_\alpha}_*(X_\alpha,Y_\alpha) \to \Fr^{S_\beta}_*(X_\beta,Y_\beta).
\]
One can show (cf. \cite{GD67}) that 
\[
\Fr^{S}_*(X,Y) = \varinjlim \Fr^{S_\alpha}_*(X_\alpha,Y_\alpha).
\]
The claim follows.
\end{proof}

\section{Geometric lemmas}

Recall that all the considered schemes are noetherian and separated, in particular, all the rings below are supposed to be noetherian.

\begin{lemma} \label{lm:Picard_rigidity}
Let $S=\Spec R$ be the spectrum of a Henselian local ring, $\relpcurve\to S$ be a projective morphism and $\cL_0, \cL_1$ be line bundles over $\relpcurve$. Denote $k=R/\mm$ the residue field and $\overline{C}=\relpcurve \times_{S} \Spec k$ the closed fiber. Suppose that there exists an isomorphism of line bundles $\cL_0|_{\overline{C}} \cong \cL_1|_{\overline{C}}$. Then for $n\in \NN$ invertible in $R$ there exists a line bundle $\cL$ over $\relpcurve$ such that $\cL^{\otimes n}\cong \cL_0\otimes \cL_1^{-1}$ and $\cL|_{\overline{C}}\cong \struct_{\overline{C}}$.
\end{lemma}
\begin{proof}
A short exact sequence of \etale sheaves
\[
1\to \mu_n \to \Gm \xrightarrow{n} \Gm \to 1
\]
gives rise to the following commutative diagram with exact rows.
\[
\xymatrix{
\coh^1_{\et}(\relpcurve, \mu_n) \ar[r] \ar[d] & \coh^1_{\et}(\relpcurve, \Gm) \ar[r]^n \ar[d] & \coh^1_{\et}(\relpcurve, \Gm) \ar[r] \ar[d] & \coh^2_{\et}(\relpcurve, \mu_n) \ar[d] \\
\coh^1_{\et}(\overline{C}, \mu_n) \ar[r] & \coh^1_{\et}(\overline{C}, \Gm) \ar[r]^n & \coh^1_{\et}(\overline{C}, \Gm) \ar[r] & \coh^2_{\et}(\overline{C}, \mu_n)
}
\]
Both the side vertical maps are isomorphisms by \cite[Chapter~VI, Corollary~2.7]{Mi80}. Identifying the middle terms with the corresponding Picard groups we obtain 
\[
\xymatrix{
\coh^1_{\et}(\relpcurve, \mu_n) \ar[r] \ar[d]^\cong & \Pic(\relpcurve) \ar[r]^n \ar[d] & \Pic(\relpcurve) \ar[r] \ar[d] & \coh^2_{\et}(\relpcurve, \mu_n) \ar[d]^{\cong} \\
\coh^1_{\et}(\overline{C}, \mu_n) \ar[r] & \Pic(\overline{C}) \ar[r]^n & \Pic(\overline{C}) \ar[r] & \coh^2_{\et}(\overline{C}, \mu_n)
}
\]
The claim follows by diagram chase.
\end{proof}

\begin{lemma} \label{lm:Serre_vanishing}
Let $S=\Spec R$ be the spectrum of a ring, $\relpcurve\to S$ be a projective morphism, $Z\subset \relpcurve$ be a closed subscheme and $\mathcal{E}, \struct(1)$ be line bundles over $\relpcurve$ with $\struct(1)$ being very ample. Then there exists $N\in \mathbb{N}$ such that for every $N'\ge N$ and $a\in\Gamma(Z,\mathcal{E}(N')|_{Z})$ there exists $s\in\Gamma(\relpcurve,\mathcal{E}(N'))$ such that $s|_Z=a$.
\end{lemma}
\begin{proof}
Let $i\colon Z\to \relpcurve$ be the closed immersion and let $\mathcal{I}$ be the sheaf of ideals defining $Z$. For every $N'\in\NN$ the short exact sequence
\[
0\to \mathcal{I} \to \struct_{\relpcurve}\to i_*\struct_{Z} \to 0
\]
of coherent $\struct_{\relpcurve}$-modules gives rise to the exact sequence of cohomology groups
\[
\coh^0_\Zar(\relpcurve,\mathcal{E}(N')) \to \coh^0_\Zar(\relpcurve,i_*\mathcal{E}|_{Z}(N')) \to \coh^1_\Zar(\relpcurve,\mathcal{E}\otimes_{\struct_{\relpcurve}} \mathcal{I}(N')).
\]
By Serre's theorem \cite[Chapter~III, Theorem~5.2]{Ha77} there exists an integer $N$ such that for every $N'\ge N$ the rightmost term vanishes, whence the claim.
\end{proof}

\begin{corollary} \label{cor:section_invertible}
Let $S=\Spec R$ be the spectrum of a local ring, $\relpcurve\to S$ be a projective morphism of relative dimension $1$, let $\mathcal{E}, \struct(1)$ be line bundles over $\relpcurve$ with $\struct(1)$ being very ample and let $Z_1,Z_2\subset \relpcurve$ be closed subschemes finite over $S$ such that $Z_1\cap Z_2=\emptyset$. Then there exists $N\in \mathbb{N}$ such that for every $N'\ge N$ there exists $s\in\Gamma(\relpcurve,\mathcal{E}(N'))$ satisfying 
\begin{enumerate}
\item
$Z(s)\cap Z_1=\emptyset$;
\item
$Z_2\subset Z(s)$;
\item
$Z(s)$ is finite over $S$.
\end{enumerate}
\end{corollary}
\begin{proof}
The scheme $Z_1$ is finite over $S$ thus semilocal. It follows that for every $n\in\mathbb{N}$ the line bundle $\mathcal{E}(n)|_{Z_1}$ is trivial whence it has a nowhere vanishing section $a_n\in\Gamma(Z_1,\mathcal{E}(n)|_{Z_1})$. Let $\overline{C}_1,\overline{C}_2,\hdots,\overline{C}_l$ be the irreducible components of the closed fiber of $\relpcurve$. Choose a collection of closed points $y_1\in \overline{C}_1-Z_2,y_2\in \overline{C}_2-Z_2,\hdots, y_l\in \overline{C}_l-Z_2$. Lemma~\ref{lm:Serre_vanishing} yields that there exists $N\in\mathbb{N}$ such that for every $N'\ge N$ there exists $s\in\Gamma(\relpcurve,\mathcal{E}(N'))$ satisfying $s|_{Z_1}=a_n$, $s|_{Z_2}=0$, $s|_{y_i}\neq 0$, $1\le i\le l$. The closed subset $Z(s)$ is projective over $S$ and its closed fiber is finite since it does not contain the irreducible components of the closed fiber of $\relpcurve$. Then $Z(s)$ is quasi-finite and projective over $S$ whence finite. The claim follows.
\end{proof}

\begin{definition} \label{def:fine_compactification}
Let $S=\Spec R$ be the spectrum of a local ring and $\relcurve\to S$ be a smooth morphism of relative dimension $1$. We say that $(\relcurve\subset\relpcurve,\struct(1))$ with $\relcurve$ being open and dense in $\relpcurve$ and $\struct(1)$ being a very ample line bundle over $\relpcurve$  is a \textit{fine compactification of $\relcurve$ over $S$} if there exists $\zeta_\infty\in \Gamma(\relpcurve,\struct(1))$ such that 
\begin{enumerate}
\item
$\relcurve=\relpcurve-Z(\zeta_\infty)$;
\item
$Z(\zeta_\infty)$ is finite over $S$.
\end{enumerate}
\end{definition}

\begin{remark}
Up to a choice of coordinates on $\A^d_S$ a fine compactification of $\relcurve\to S$ is a closed embedding $\relcurve\to \A^d_S$ such that $\relpcurve-\relcurve$ is finite over $S$ with $\relpcurve$ being the projective closure of $\relcurve$.
\end{remark}

\begin{lemma} \label{lm:compactify_subcurve}
Let $S=\Spec R$ be the spectrum of a local ring, $\relcurve\to S$ be a smooth morphism of relative dimension $1$ and $\relcurve'\subset \relcurve$ be a Zariski neighborhood of a closed point $y\in \relcurve$ such that $\relcurve-\relcurve'$ is quasi-finite over $S$. Suppose that there exists an dense open immersion $\relcurve\subset \relpcurve$ with $\relpcurve$ being projective over $S$ and $\relpcurve-\relcurve$ being finite over $S$. Then there exists a Zariski neighborhood $\relcurve''\subset \relcurve'$ of $y$ admitting a fine compactification.
\end{lemma}
\begin{proof}
Let $\struct(1)$ be a very ample line bundle over $\relpcurve$. By the assumption $D=\relpcurve-\relcurve'$ is quasi-finite over $S$ whence finite. Corollary~\ref{cor:section_invertible} yields that there exists $N\in\mathbb{N}$ and $\zeta_\infty\in\Gamma(\relpcurve,\struct(N))$ such that $D\subset Z(\zeta_\infty)$, $\zeta_{\infty}|_y\neq 0$ and $Z(\zeta_\infty)$ is finite over $S$. Set $\relcurve''=\relpcurve-Z(\zeta_\infty)$. It follows that $(\relcurve''\subset \relpcurve,\struct(N))$ is a fine compactification of $\relcurve''$.
\end{proof}

\begin{lemma} \label{lm:sections}
Let $S=\Spec R$ be the spectrum of a Henselian local ring, $\relcurve\to S$ be a smooth morphism of relative dimension $1$ and $r_0,r_1\colon S\to \relcurve$ be morphisms of $S$-schemes such that $r_0(x)=r_1(x)$ for the closed point $x\in S$. Let $(\relcurve\subset \relpcurve,\struct(1)) $ be a fine compactification of $\relcurve$, set $D=\relpcurve-\relcurve$ and denote $\overline{C}=\relpcurve \times_{S} \{x\}$ the closed fiber. For $\cL_0=\struct_{\relpcurve}(r_0(S)),\, \cL_1=\struct_{\relpcurve}(r_1(S))$ and an integer $n\in\NN$ invertible in $R$ let $\cL$ be a line bundle over $\relpcurve$ and $\Theta\colon\cL_0\cong \cL_1\otimes \cL^{\otimes n}$, $\theta\colon \cL|_{\overline{C}}\xrightarrow{\simeq} \struct_{\overline{C}}$ be isomorphisms given by Lemma~\ref{lm:Picard_rigidity}. Then there exists $N\in \mathbb{N}$ and sections
\[
s_0\in\Gamma(\relpcurve,\cL_0),\quad s_1\in\Gamma(\relpcurve,\cL_1),\quad  \zeta\in\Gamma(\relpcurve,\cL_0(nN)),\quad \xi_0,\tau_0\in\Gamma(\relpcurve,\struct(N)),\quad \xi_1,\tau_1\in \Gamma(\relpcurve,\cL(N))
\]
such that
\begin{enumerate}
\item
$Z(s_0)=r_0(S)$, $Z(s_1)=r_1(S)$;
\item
$D\subset Z(\zeta)$, $Z(s_0)\cap Z(\zeta)=Z(s_1)\cap Z(\zeta)=\emptyset$;
\item
$Z(\xi_0)\cap Z(s_0) = Z(\xi_1)\cap Z(s_1) = Z(\xi_0)\cap Z(\zeta)=Z(\xi_1)\cap Z(\zeta)=\emptyset$;
\item
$(s_0\otimes \xi_0^{\otimes n})|_{\overline{C}} = \Theta(nN)^* ( s_1\otimes \xi_1^{\otimes n})|_{\overline{C}}$;
\item
$(s_0\otimes \xi_0^{\otimes n})|_{Z(\zeta)}= \Theta(nN)^*(s_1\otimes \xi_1^{\otimes n})|_{Z(\zeta)}$;
\item
$Z(\zeta)\cup Z(s_0)\subset Z(\tau_0)$, $Z(\zeta)\cup Z(s_1)\subset Z(\tau_1)$, $Z(\xi_0)\cap Z(\tau_0)=Z(\xi_1)\cap Z(\tau_1)=\emptyset$;
\item
the restriction map $\Gamma(\relpcurve,\struct(N))\to \Gamma(Z(\zeta)\cup Z(s_0)\cup\overline{C},\struct(N)|_{Z(\zeta)\cup Z(s_0)\cup\overline{C}})$ is surjective;
\item
$Z(\zeta), Z(\xi_0), Z(\xi_1), Z(\tau_0)$ and $Z(\tau_1)$ are finite over $S$.
\end{enumerate}
Here $Z(s)$ is the vanishing locus of a section $s$.
\end{lemma}

\begin{proof}
\textbf{(a) $\mathbf{s_0}$ and $\mathbf{s_1}$ compatible over $\mathbf{\overline{C}}$.} The line bundles $\cL_0,\cL_1$ have canonical sections $s_0\in\Gamma(\relpcurve,\cL_0) ,s_1'\in\Gamma(\relpcurve,\cL_1)$ such that $Z(s_0)=r_0(S)$ and $Z(s_1')=r_1(S)$. The isomorphisms $\Theta$ and $\theta$ induce an isomorphism 
\[
\id\otimes \theta^{\otimes n}\circ\Theta|_{\overline{C}}\colon \cL_0|_{\overline{C}}\xrightarrow{\simeq} \cL_1|_{\overline{C}}.
\]
By the assumption $\divisor_0 s_0|_{\overline{C}}=\divisor_0 s_1'|_{\overline{C}}$ whence $s_0|_{\overline{C}}=\overline{\alpha}(\id\otimes \theta^{\otimes n}\circ\Theta|_{\overline{C}})^*s_1'|_{\overline{C}}$ for some $\overline{\alpha}\in (R/\mm)^*$ where $\mm$ is the maximal ideal of $R$. Choose $\alpha\in R^*$ such that $\alpha = \overline{\alpha}\mmod \mm$ and set $s_1=\alpha s_1'$. By the above we have
\[
Z(s_0)=r_0(S),\quad Z(s_1)=r_1(S),\quad s_0|_{\overline{C}}=(\id\otimes \theta^{\otimes n}\circ\Theta|_{\overline{C}})^*s_1|_{\overline{C}}.
\]

\textbf{(b) $\mathbf{\zeta}'$ and $\mathbf{\zeta_{\infty}}$.} By Corollary~\ref{cor:section_invertible} there exists $N_1\in\mathbb{N}$ and $\zeta'\in\Gamma(\relpcurve,\cL_0(nN_1))$ such that 
\[
Z(s_0)\cap Z(\zeta')=Z(s_1)\cap Z(\zeta')=\emptyset.
\]
Choose $\zeta_\infty\in\Gamma(\relpcurve,\struct(1))$ such that $Z(\zeta_\infty)^\red=D$.

\textbf{(c) $\mathbf{\xi_0}'$, $\mathbf{\xi_1'}$ and $\mathbf{\xi}$.} By Lemma~\ref{lm:Serre_vanishing} and Corollary~\ref{cor:section_invertible} there exists $N_2\in\mathbb{N}$ such that the restriction 
\[
\Gamma(\relpcurve,\struct(N_2))\to \Gamma(Z(\zeta')\cup D\cup\overline{C}, \struct(N_2)|_{Z(\zeta')\cup D\cup\overline{C}})
\]
is surjective and there exists $\xi_1'\in \Gamma(\relpcurve,\cL(N_2))$ satisfying
\[
Z(\xi_1')\cap Z(s_0)=Z(\xi_1')\cap Z(s_1)=Z(\xi_1')\cap D=Z(\xi_1')\cap Z(\zeta')=\emptyset.
\]
The scheme $Z(\zeta)\cup D$ is finite over $S$ thus it is the spectrum of a direct product of Henselian local rings \cite[Theorem~4.2, Corollary~4.3]{Mi80}. Since by the construction $s_0$ and $s_1$ coincide on the closed fiber (up to the isomorphism of the corresponding bundles) then it follows that there exists $a\in \Gamma(Z(\zeta')\cup D, \struct(N_2)|_{Z(\zeta')\cup D})$ such that
\[
a|_{(Z(\zeta')\cup D)\cap\overline{C}}=((\theta^{-1})^*\xi_1'|_{\overline{C}})|_{(Z(\zeta')\cup D)\cap\overline{C}}, \quad s_0|_{Z(\zeta')\cup D} \otimes a^{\otimes n} = \Theta(nN)^*(s_1\otimes (\xi_1')^{\otimes n})|_{Z(\zeta')\cup D}.
\]
Choose $\xi_0'\in \Gamma(\relpcurve,\struct(N_2))$ such that $\xi_0'|_{\overline{C}}=(\theta^{-1})^*\xi_1'|_{\overline{C}}$ and $\xi_0'|_{Z(\zeta')\cup D}=a$. Then
\[
(s_0\otimes (\xi'_0)^{\otimes n})|_{\overline{C}} = \Theta(nN)^* ( s_1\otimes (\xi'_1)^{\otimes n})|_{\overline{C}},\quad
(s_0\otimes (\xi'_0)^{\otimes n})|_{Z(\zeta)\cup D}= \Theta(nN)^*(s_1\otimes (\xi'_1)^{\otimes n})|_{Z(\zeta)\cup D}.
\]
It follows that $Z(\xi_0')\cap Z(s_0)=Z(\xi_0')\cap Z(s_1)=Z(\xi_0')\cap D=Z(\xi_0')\cap Z(\zeta')=\emptyset$ since the same holds for $\xi_1'$ and the closed points of the intersections belong to $\overline{C}$. By Corollary~\ref{cor:section_invertible} there exists $N_3\in\mathbb{N}$ and $\xi\in\Gamma(\relpcurve,\struct(N_3))$ such that 
\[
Z(\xi)\cap D =Z(\xi)\cap Z(\zeta')=Z(\xi)\cap Z(s_0)=Z(\xi)\cap Z(s_1)=\emptyset.
\]

\textbf{(d) $\mathbf{N}$, $\mathbf{\tau_0'}$ and $\mathbf{\tau_1'}$.}
By Lemma~\ref{lm:Serre_vanishing} and Corollary~\ref{cor:section_invertible} there exists $N\ge \max(N_1,N_2)$ such that $N= N_2\mmod N_3$, the restriction 
\[
\Gamma(\relpcurve,\struct(N))\to \Gamma(Z(\zeta')\cup D\cup Z(s_0)\cup \overline{C},\struct(N)|_{Z(\zeta')\cup D \cup Z(s_0)\cup\overline{C}})
\]
is surjective and there exist 
\[
\tau_0'\in\Gamma(\relpcurve,\cL_0^{-1}\otimes\cL_0(nN_1)^{-1}\otimes\struct(N-1)),\quad \tau_1'\in\Gamma(\relpcurve,\cL_1^{-1}\otimes\cL_0(nN_1)^{-1}\otimes\cL(N-1))
\]
satisfying
\[
Z(\xi_0')\cap Z(\tau_0')=Z(\xi)\cap Z(\tau_0')=Z(\xi_1')\cap Z(\tau_1')=Z(\xi)\cap Z(\tau_1')=\emptyset.
\]

\textbf{(e) $\mathbf{\zeta}$, $\mathbf{\xi_0}$, $\mathbf{\xi_1}$, $\mathbf{\tau_0}$ and $\mathbf{\tau_1}$.} Set
\begin{gather*}
\zeta=\zeta'\otimes \zeta_\infty^{\otimes n(N-N_1)},\quad\xi_0=\xi_0'\otimes \xi^{\otimes \tfrac{N-N_2}{N_3}},\quad \xi_1=\xi_1'\otimes \xi^{\otimes \tfrac{N-N_2}{N_3}},\\
\tau_0=s_0\otimes\zeta'\otimes\tau_0'\otimes \zeta_\infty,\quad \tau_1=s_1\otimes\zeta'\otimes\tau_1'\otimes \zeta_\infty.
\end{gather*}
The zero loci of the sections constructed with the usage of Corollary~\ref{cor:section_invertible} are finite over $S$. The claim follows.
\end{proof}

\begin{lemma} \label{lm:construct_compactification}
Let $k$ be a field and $p\colon U\to \A^d_k$ be an \etale neighborhood of $\{0\}$. Set 
\[
S_d=\Spec \struct_{\A^d_k,0}^h,\quad S_{d-1}=\Spec \struct_{\A^{d-1}_k,0}^h,
\]
let $r_d\colon S_{d-1}\to S_d$ be the morphism induced by the embedding $\A^{d-1}_k=\A^{d-1}_k\times\{0\}\subset \A^d_k$ and $j\colon S_d\to U$ be the structure morphism. Then there exists a smooth morphism $\rho\colon\cU\to S_{d-1}$ of relative dimension $1$, a morphism $f\colon \cU\to U$ and a morphism $g\colon S_d\to \cU$ such that
\begin{enumerate}
\item
$\cU$ admits a fine compactification over $S_{d-1}$;
\item
$j=f\circ g$ and $\rho\circ g\circ r_d = \id_{S_{d-1}}$.
\end{enumerate}
\[
\xymatrix @C=3pc{
S_d \ar[r]_(0.6){g} \ar@/^1.0pc/[rrr]^{j} & \cU \ar[rr]_{f} \ar[d]^{\rho} & & U \ar[d]^p \\
& S_{d-1} \ar[ul]^{r_d} & & \A^d_k 
}
\]
\end{lemma}
\begin{proof}
Without loss of generality we may assume that $U$ is connected. Let $V$ be an open subset of $\A^d_k$ such that $U\times_{\A^d_k} V\to V$ is finite and set $Z=\A^d_k-V$. Choose some generators $\{h_1,h_2,\hdots,h_l\}$ defining the ideal of $Z$ and set $N=\deg h_1$. Consider the curve
\[
K_N=Z(x_1-x_d^{N},x_2-x_d^{N^2},\hdots, x_{d-1}-x_d^{N^{d-1}})\subset \A^d_k.
\]
It is straightforward to see that $h_1(x_d^{N},x_d^{N^2},\hdots,x_{d-1}^{N^{d-1}}, x_d)\neq 0$. Thus $K_N\not\subset Z$ and $K_N\cap V\neq\emptyset$. Changing the coordinates on $\A^d_k$ as 
\[
(x_1,x_2,\hdots,x_d)\mapsto (x_1-x_d^{N},x_2-x_d^{N^2},\hdots, x_{d-1}-x_d^{N^{d-1}},x_d)
\]
we may assume that $Z(x_1,x_2,\hdots,x_{d-1})\cap V\neq \emptyset$.

The \etale morphism $p\colon U\to \A^d_k$ induces an \etale morphism $\widetilde{p}\colon U\to \A^{d-1}_k\times \PP^1_k$. Applying Zariski's Main Theorem \cite[Theorem~8.12.6]{GD67} we obtain a factorization $\widetilde{p}=\overline{p}\circ \widetilde{p}'$,
\[
U\xrightarrow{\widetilde{p}'} X \xrightarrow{\overline{p}} \A^{d-1}_k\times \PP^1_k,
\]
with $\widetilde{p}'$ being an open embedding and $\overline{p}$ being finite. We may assume that $U$ is dense in $X$. Consider the base change of the above factorization with respect to the structure morphism $S_{d-1}\to \A^{d-1}_k$,
\[
U_S \xrightarrow{q'} X_S \xrightarrow{\overline{q}'} S_{d-1}\times \PP^1_k,
\]
where $U_S=S_{d-1}\times_{\A^{d-1}_k} U$ and $X_S=S_{d-1}\times_{\A^{d-1}_k} X$. The projection $S_d\to S_{d-1}$ together with the structure map $j\colon S_d\to U$ induce a morphism $g'\colon S_d\to U_S$ such that 
\[
f' \circ g'= j\,\quad r_d\circ g' \circ \rho'=\id_{S_{d-1}}
\]
for the projections $f'\colon U_S=S_{d-1}\times_{\A^{d-1}_k} U\to U$ and $\rho'\colon U_S=S_{d-1}\times_{\A^{d-1}_k} U\to S_{d-1}$.
\[
\xymatrix @C=3pc{
S_d \ar[r]_(0.6){g'} \ar@/^1.0pc/[rrr]^{j} & U_S \ar[rr]_{f'} \ar[d]^{\rho'} \podr & & U \ar[d] \\
& S_{d-1} \ar[ul]^{r_d}\ar[rr] & & \A^{d-1}_k 
}
\]
The morphism $\rho'$ is smooth of relative dimension 1 being the base change of such morphism. It remains to construct a fine compactification of a Zariski neighborhood of $g'(\{0\})$ in $U_S$.

The scheme $X_S$ being finite over $S_{d-1}\times \PP^1_k$ is projective over $S_{d-1}$. The composition $\overline{q}'\circ q'$ is finite over $V_S=S_{d-1}\times_{\A^{d-1}_k} V$, hence over $V_S$ morphism $q'$ is open and finite whence an isomorphism. Thus $q'(D_S)\cap V_S=\emptyset$ for $D_S=X_S-U_S$. Recall that $(\{0\}\times \PP^1_k)\cap V\neq \emptyset$, whence $\{0\}\times\PP^1_k\not\subset\overline{q}'(D_S)$. It follows that the closed fiber of $D_S$ is finite. Since $D_S$ is projective over $S_{d-1}$ then $D_S$ is finite over $S_{d-1}$. Lemma~\ref{lm:compactify_subcurve} applied to $U_S\subset X_S$ yields that there exists a Zariski neighborhood $\cU\subset U_S$ of $g'(\{0\})$ admitting a fine compactification. The claim follows.
\end{proof}

\begin{lemma} \label{lm:framing}
Let $S=\Spec R$ be the spectrum of a local ring, $\relcurve \to S$ be a smooth morphism of relative dimension $1$ and $y\in \relcurve$ be a closed point. Suppose that $\relcurve$ admits a fine compactification. Then there exists a Zariski neighborhood $\relcurve' \subset \relcurve$ of $y$ admitting a fine compactification and a normal framing (see Definition~\ref{def:normal_framing}).
\end{lemma}
\begin{proof}
Let $(\relcurve\subset \relpcurve,\struct(1))$ be a fine compactification of $\relcurve$. Consider the closed immersion $\relpcurve\to \PP^{m+1}_S$ given by $\struct(1)$. Changing the coordinates on $\PP^{m+1}_S$ we may assume that $\relcurve = \relpcurve \cap \A^{m+1}_S$.

Let $C_1,C_2,\hdots,C_l$ be the connected components of the closed fiber of $\relcurve$. Choose some closed points $y_1\in C_1,y_2\in C_2,\hdots,y_l\in C_l$. Since $\relcurve\to S$ is smooth at $y,y_1,y_2,\hdots,y_l$ then there exists an affine open subscheme $W\subset \A^{m+1}_S$ and regular functions $\psi_1,\psi_2,\hdots,\psi_m \in\Gamma(W,\struct_{W})$ such that $y,y_1,y_2,\hdots,y_l\in W$ and $\relcurve\cap W=Z(\psi_1,\psi_2,\hdots,\psi_m)$. Set $\relcurve'=\relcurve\cap W$. We have $\relcurve'\cap C_1\neq\emptyset, \relcurve'\cap C_2\neq\emptyset,\hdots, \relcurve'\cap C_l\neq\emptyset$, thus $\relcurve-\relcurve'$ is quasi finite over $S$.

Let $N_{W/{\relcurve'}}$ be the normal bundle for the closed immersion $\relcurve'\to W$. Choose a splitting $r\colon N_{W/{\relcurve'}}\to T_W|_{\relcurve'}$ for the epimorphism $T_W|_{\relcurve'} \to N_{W/{\relcurve'}}$. The splitting $r$ gives rise to a closed immersion of the total spaces $N_{W/{\relcurve'}} \to T_W$. Consider the morphism 
\[
p\colon N_{W/{\relcurve'}}\to \A^{m+1}_S
\]
given by the composition
\[
N_{W/{\relcurve'}} \to T_W \to T_{\A^{m+1}_S} \xrightarrow{\simeq}  \A^{m+1}_S \times_S \A^{m+1}_S\to \A^{m+1}_S.
\]
Here the first map is the closed embedding induced by $r$, the second one is given by the immersion $T_W\subset T_{\A^{m+1}_S}$, the third one is the canonical trivialization $T_{\A^{m+1}_S}\cong \struct^{\oplus (m+1)}_{\A^{m+1}_S}$ and the last one is the addition morphism. 

Let $z\colon \relcurve'\to N_{W/\relcurve'}$ be the zero section. Then there is a decomposition
\[
\left.\left(T_{(N_{W/\relcurve'})}\right)\right|_{z(\relcurve')}\cong T_{\relcurve'} \oplus N_{W/\relcurve'}
\]
and it is straightforward to check that $p$ induces an isomorphism
\[
\left.\left(T_{(N_{W/\relcurve'})}\right)\right|_{z(\relcurve')}\cong T_{\relcurve'} \oplus  N_{W/\relcurve'} \cong \left.T_{\A^{m+1}_S}\right|_{\relcurve'}.
\]
Hence $p$ is \etale at the points of $z(\relcurve')$. Let $\widetilde{W}'$ be the neighborhood of $z(\relcurve')$ where $p$ is \'etale. We have $p(z(\relcurve'))=\relcurve'$ thus $p^{-1}(\relcurve')=z(\relcurve')\sqcup Z$. Set $\widetilde{W}=\widetilde{W}'\cap p^{-1}(W)-Z$, let $p'\colon \widetilde{W}\to W$ be the restriction of the projection $p$ and let $\rho \colon \widetilde{W}\to \relcurve'$ be the projection induced by the projection $N_{W/\relcurve'}\to \relcurve'$. Then
\[
(W, p'\colon \widetilde{W}\to W,  (\psi_1\circ p', \psi_2\circ p', \hdots, \psi_m\circ p'), \rho)
\]
is a level $m$ normal framing of $\relcurve'$.

The claim follows by Lemma~\ref{lm:compactify_subcurve} since an open subscheme of a scheme admitting a normal framing clearly admits a normal framing.
\end{proof}

\section{Rigidity for linear framed presheaves}

\begin{theorem} \label{thm:main}
Let $S=\Spec R$ be the spectrum of a Henselian local ring, $\relcurve\to S$ be a smooth morphism of relative dimension $1$ admitting a fine compactification and $r_0,r_1\colon S\to \relcurve$ be morphisms of $S$-schemes such that $r_0(x)=r_1(x)$ for the closed point $x\in S$. Then for every $n\in \mathbb{N}$ such that $n\in R^*$ the following holds.
\begin{enumerate}
\item
If $2\in R^*$ then
\[
\la \sigma_{\relcurve}^m \ra \circ r_1 - \la \sigma_{\relcurve}^m \ra \circ r_0 = H \circ i_1 - H \circ i_0   + n h_{\relcurve}\circ a
\]
for some $m\in \mathbb{N}$, $H\in \ZF_m^S(\A^1\times S,\relcurve)$ and $a\in \ZF_{m-1}^S(S,\relcurve)$.
\item
If $2=0$ in $R$ then 
\[
\la \sigma_{\relcurve}^m \ra \circ r_1 - \la \sigma_{\relcurve}^m \ra \circ r_0 = H \circ i_1 - H \circ i_0   + n \sigma_\relcurve \circ a
\]
for some $m\in \mathbb{N}$, $H\in \ZF_m^S(\A^1\times S,\relcurve)$ and $a\in \ZF_{m-1}^S(S,\relcurve)$.
\end{enumerate}
Here $i_0,i_1\colon S\to \A^1\times S$ are the closed immersions given by $\{0\}\times S$ and $\{1\}\times S$ respectively.
\end{theorem}

\begin{proof}
We give the detailed proof only for the first claim. The reasoning for the second one is literally the same up to the usage of $n$ instead of $2n$ and the observation that $h_{\relcurve}=2\sigma_{\relcurve}$ if $2=0$.

We need to show that
\[
\la \sigma_{\relcurve}^m \ra \circ r_1 - \la \sigma_{\relcurve}^m \ra \circ r_0 \simA n h_{\relcurve}\circ a
\]
for some $m\in \mathbb{N}$ and $a\in \ZF_{m-1}^S(S,\relcurve)$. Set $y=r_0(x)=r_1(x)$. It is sufficient to prove the theorem for a Zariski neighborhood $\relcurve'$ of $y$ admitting a fine compactification. Indeed, composing the claim for $\relcurve'$ with the open immersion $j\colon \relcurve'\to \relcurve$ and applying Lemma~\ref{lm:basic_homotopies2}(\ref{lm:basic_homotopies2:shuffle}) one obtains the claim for $\relcurve$. Thus Lemma~\ref{lm:framing} yields that we may assume that $\relcurve$ admits a normal framing as well as a fine compactification, i.e. that $\lF_{m-2}(\relcurve) \neq \emptyset$ for some $m\in\NN$.

Let $(\relcurve\subset \relpcurve,\struct(1))$ be a fine compactification of $\relcurve$. Denote $Z_0=r_0(S)\subset \relcurve\subset\relpcurve$ and $Z_1=r_1(S)\subset \relcurve\subset\relpcurve$ and let $\cL_0=\struct_{\relpcurve}(Z_0)$ and $\cL_1=\struct_{\relpcurve}(Z_1)$ be the corresponding line bundles. Applying Lemmas~\ref{lm:Picard_rigidity} and~\ref{lm:sections} we obtain a line bundle $\cL$ over $\relpcurve$, an isomorphism $\Theta\colon \cL_0\xrightarrow{\simeq}\cL_1\otimes \cL^{\otimes 2n}$, a natural number $N\in\mathbb{N}$ and sections
\[
s_0\in\Gamma(\relpcurve,\cL_0),\quad s_1\in\Gamma(\relpcurve,\cL_1),\quad  \zeta\in\Gamma(\relpcurve,\cL_0(2nN)),\quad \xi_0,\tau_0\in\Gamma(\relpcurve,\struct(N)),\quad \xi_1,\tau_1\in \Gamma(\relpcurve,\cL(N))
\]
such that
\begin{enumerate}
\item
$Z(s_0)=Z_0$, $Z(s_1)=Z_1$;
\item
$D\subset Z(\zeta)$, $Z(s_0)\cap Z(\zeta)=Z(s_1)\cap Z(\zeta)=\emptyset$;
\item
$Z(\xi_0)\cap Z(s_0) = Z(\xi_1)\cap Z(s_1) = Z(\xi_0)\cap Z(\zeta)=Z(\xi_1)\cap Z(\zeta)=\emptyset$;
\item
$(s_0\otimes \xi_0^{\otimes 2n})|_{\overline{C}} = \Theta(2nN)^* ( s_1\otimes \xi_1^{\otimes 2n})|_{\overline{C}}$;
\item
$(s_0\otimes \xi_0^{\otimes 2n})|_{Z(\zeta)}= \Theta(2nN)^*(s_1\otimes \xi_1^{\otimes 2n})|_{Z(\zeta)}$;
\item
$Z(\zeta)\cup Z(s_0)\subset Z(\tau_0)$, $Z(\zeta)\cup Z(s_1)\subset Z(\tau_1)$, $Z(\xi_0)\cap Z(\tau_0)=Z(\xi_1)\cap Z(\tau_1)=\emptyset$;
\item
the restriction map $\Gamma(\relpcurve,\struct(N))\to \Gamma(Z(\zeta)\cup Z(s_0)\cup \overline{C},\struct(N)|_{Z(\zeta)\cup Z(s_0)\cup \overline{C}})$ is surjective;
\item
$Z(\zeta), Z(\xi_0), Z(\xi_1), Z(\tau_0)$ and $Z(\tau_1)$ are finite over $S$.
\end{enumerate}
Here $\overline{C}=\relpcurve\times_S \{x\}$ is the closed fiber.

Let $\pi_{\relpcurve}\colon \A^1\times \relpcurve\to \relpcurve$ be the projection and consider the following sections of $\pi_{\relpcurve}^* \cL_0(2nN)$.
\begin{gather*}
\Upsilon_0=\pi_{\relpcurve}^* \zeta,\quad  \Upsilon_1=t\pi_{\relpcurve}^*(s_0\otimes \xi_0^{\otimes 2n})+(1-t)\pi_{\relpcurve}^*\Theta(2nN)^* (s_1 \otimes \xi_1^{\otimes 2n}) \in \Gamma(\A^1\times \relpcurve,\pi_{\relpcurve}^* \cL_0(2nN)).
\end{gather*}
Here $t$ is the coordinate function on $\A^1$. We have $Z(\Upsilon_0)  = \A^1\times Z(\zeta)$ and by the property (5) above
\[
\Upsilon_1|_{\A^1\times Z(\zeta)}=\pi_{Z(\zeta)}^*(s_0\otimes \xi_0^{\otimes 2n})|_{Z(\zeta)}
\]
with $\pi_{Z(\zeta)}\colon \A^1\times Z(\zeta)\to Z(\zeta)$ being the projection. Thus it follows from the properties (2) and (3) that
\[
Z(\Upsilon_0)\cap Z(\Upsilon_1)=\emptyset.
\] 
Consider the morphisms
\[
[\Upsilon_1:\Upsilon_0]\colon \A^1\times \relpcurve \to \PP^1_S, \quad \frac{\Upsilon_1}{\Upsilon_0}\colon \A^1\times (\relpcurve - Z(\zeta)) \to \A^1.
\]
One can easily see that
\[
(\pi_{\A^1},[\Upsilon_1:\Upsilon_0])\colon \A^1\times \relpcurve \to \A^1\times\PP^1_S
\]
is quasi-finite and proper whence finite. Thus $Z(\Upsilon_1)$ is finite over $\A^1_S$. 

It follows from the property (2) above that $\relpcurve-Z(\zeta)\subset \relcurve$. Set $j\colon \relpcurve-Z(\zeta)\to \relcurve$ for the open immersion. The framed correspondence
\[
(Z(\Upsilon_1),\A^1\times (\relpcurve - Z(\zeta)),\tfrac{\Upsilon_1}{\Upsilon_0},j\circ \pi_{\relpcurve - Z(\zeta)}) \in \Fr_\relcurve^S(\A^1\times S,\relcurve)
\]
with $\pi_{\relpcurve - Z(\zeta)}\colon \A^1\times (\relpcurve - Z(\zeta)) \to \relpcurve - Z(\zeta)$ being the projection yields an equivalence
\[
(Z(s_0\otimes \xi_0^{\otimes 2n}), \relpcurve - Z(\zeta), \frac{s_0\otimes \xi_0^{\otimes 2n}}{\zeta}, j) \simA
(Z(s_1\otimes \xi_1^{\otimes 2n}), \relpcurve - Z(\zeta), \frac{\Theta(2nN)^*(s_1\otimes \xi_1^{\otimes 2n})}{\zeta}, j) \in \Fr^S_{\relcurve}(S,\relcurve).
\]

Since $Z(s_0)\cap Z(\xi_0)=Z(s_1)\cap Z(\xi_1)=\emptyset$ by the property (3) above we have
\begin{multline*}
\la Z(s_0\otimes \xi_0^{\otimes 2n}), \relpcurve - Z(\zeta), \frac{s_0\otimes \xi_0^{\otimes 2n}}{\zeta}, j\ra =\\
=\la Z(s_0), \relpcurve - Z(\zeta)- Z(\xi_0), \frac{s_0\otimes \xi_0^{\otimes 2n}}{\zeta}, j\ra +
\la Z(\xi_0), \relpcurve - Z(\zeta) - Z(s_0), \frac{s_0\otimes \xi_0^{\otimes 2n}}{\zeta}, j\ra,
\end{multline*}
\begin{multline*}
\la Z(s_1\otimes \xi_1^{\otimes 2n}), \relpcurve - Z(\zeta), \frac{\Theta(2nN)^*(s_1\otimes \xi_1^{\otimes 2n})}{\zeta}, j \ra =\\
=\la Z(s_1), \relpcurve - Z(\zeta)- Z(\xi_1), \frac{\Theta(2nN)^*(s_1\otimes \xi_1^{\otimes 2n})}{\zeta}, j\ra + \\
+ \la Z(\xi_1), \relpcurve - Z(\zeta) - Z(s_1), \frac{\Theta(2nN)^*(s_0\otimes \xi_0^{\otimes 2n})}{\zeta}, j\ra
\end{multline*}
in $\ZF^S_{\relcurve}(S,\relcurve)$. Here we abuse the notation and denote by $j$ the respective open immersions. 

It follows from the property (6) above that we have
\[
\la Z(\xi_0), \relpcurve - Z(\zeta) - Z(s_0), \frac{s_0\otimes \xi_0^{\otimes 2n}}{\zeta}, j\rangle
=\la Z(\xi_0), \relpcurve - Z(\tau_0), \frac{s_0\otimes \tau_0^{\otimes 2n}}{\zeta} \left(\frac{\xi_0}{\tau_0}\right)^{2n}, j\rangle
\]
and that $\tfrac{s_0\otimes \tau_0^{\otimes 2n}}{\zeta}$ is an invertible regular function on $\relpcurve - Z(\tau_0)$. The morphism $[\xi_0:\tau_0]\colon \relpcurve\to \PP^1_S$ is projective and quasi-finite whence finite, thus $\tfrac{\xi_0}{\tau_0}\colon\relpcurve - Z(\tau_0) \to \A^1_S$ is finite as well. Thus Lemma~\ref{lm:hyperbolic} yields that for every $\Phi\in \lF_{m-2}(\relcurve)$ we have
\begin{multline*}
\sigma_{\relcurve} \circ (\Phi\frp  \la Z(\xi_0), \relpcurve - Z(\zeta) - Z(s_0), \frac{s_0\otimes \xi_0^{\otimes 2n}}{\zeta}, j\rangle) = \\ =  \sigma_{\relcurve} \circ (\Phi\frp \la Z(\xi_0), \relpcurve - Z(\tau_0), \frac{s_0\otimes \tau_0^{\otimes 2n}}{\zeta} \left(\frac{\xi_0}{\tau_0}\right)^{2n}, j\rangle)  \simA nh_\relcurve \circ a_0
\end{multline*}
for some $a_0\in \ZF_{m-1}^S(S,\relcurve)$. A similar reasoning shows that for every $\Phi\in \lF_{m-2}(\relcurve)$ we have
\[
\sigma_{\relcurve} \circ (\Phi\frp  \la Z(\xi_1), \relpcurve - Z(\zeta) - Z(s_1), \frac{\Theta(2nN)^* (s_1\otimes \xi_1^{\otimes 2n})}{\zeta}, j\rangle) \simA nh_\relcurve \circ a_1
\]
for some $a_1\in \ZF_{m-1}^S(S,\relcurve)$. Combining the above we obtain the for every $\Phi\in \lF_{m-2}(\relcurve)$ we have
\begin{multline*}
\sigma_\relcurve \circ \left (\Phi \frp \la Z(s_1), \relpcurve - Z(\zeta)- Z(\xi_1), \frac{\Theta(2nN)^*(s_1\otimes \xi_1^{\otimes 2n})}{\zeta}, j\ra - \right. \\
\left. - \Phi \frp \la Z(s_0), \relpcurve - Z(\zeta)- Z(\xi_0), \frac{s_0\otimes \xi_0^{\otimes 2n}}{\zeta}, j\ra \right) \simA n h_\relcurve \circ (a_0-a_1).
\end{multline*}. 

It remains to show that for some $\Phi\in \lF_{m-2}(\relcurve)$ the left-hand side of the above formula is equivalent to $\sigma^{m}_{\relcurve} \circ r_1 - \sigma^{m}_{\relcurve} \circ r_0 $. Choose a normal framing 
\[
\Phi' = (W\subset \A^{m-1}_S,i\colon \relcurve\to \widetilde{W},(\psi_1,\psi_2,\hdots,\psi_{m-2}),\rho) \in \lF_{m-2}(\relcurve)
\]
and set
\[
\alpha=\Jac(\Phi'\frp ( Z(s_1), \relpcurve - Z(\zeta)- Z(\xi_1), \frac{\Theta(2nN)^*(s_1\otimes \xi_1^{\otimes 2n})}{\zeta}, j)).
\]
Since $Z(\xi_1)\cap Z(s_1)=\emptyset$ by the property (3) above we have $\alpha\in R^*$. Set
\[
\Phi = (W\subset \A^{m-1}_S,i\colon \relcurve\to \widetilde{W},(\psi_1,\psi_2,\hdots,\alpha^{-1}\psi_{m-2}),\rho) \in \lF_{m-2}(\relcurve).
\]
Then $\Jac(\Phi\frp ( Z(s_1), \relpcurve - Z(\zeta)- Z(\xi_1), \frac{\Theta(2nN)^*(s_1\otimes \xi_1^{\otimes 2n})}{\zeta}, j))=1$ and Lemma~\ref{lm:suspension} yields
\[
\Phi \frp \la Z(s_1), \relpcurve - Z(\zeta)- Z(\xi_1), \frac{\Theta(2nN)^*(s_1\otimes \xi_1^{\otimes 2n})}{\zeta}, j\ra \simA \sigma^{m-1}_{\relcurve} \circ r_1.
\]
In view of the property (4) above we have
\[
\Jac(\Phi\frp ( Z(s_0), \relpcurve - Z(\zeta)- Z(\xi_0), \frac{s_0\otimes \xi_0^{\otimes 2n}}{\zeta}, j))=1 \mmod \mm_x
\]
for the maximal ideal $\mm_x$ of $R$. Then there exists $\beta\in R^*$ such that 
\[
\beta = 1 \mmod \mm_x,\quad \beta^{2n}=\Jac(\Phi\frp ( Z(s_0), \relpcurve - Z(\zeta)- Z(\xi_0), \frac{s_0\otimes \xi_0^{\otimes 2n}}{\zeta}, j)).
\]
Applying property (7) above choose $\widetilde{\xi}_0\in\Gamma(\relpcurve,\struct(N))$ such that
\[
\widetilde{\xi}_0|_{Z(\zeta)\cup \overline{C}}=\xi_0|_{Z(\zeta)\cup \overline{C}},\quad \widetilde{\xi}_0|_{Z(s_0)}=\beta^{-1}\xi_0|_{Z(s_0)}.
\]
One immediately sees that the properties (3)-(6) hold for $\widetilde{\xi}_0$ since two closed subsets of $\relpcurve$ intersect each other if and only if they intersect each other in the closed fiber $\overline{C}$. Thus all the above reasoning is valid with $\widetilde{\xi}_0$ substituted for $\xi_0$. We have 
\[
\Jac(\Phi\frp ( Z(s_0), \relpcurve - Z(\zeta)- Z(\widetilde{\xi}_0), \frac{s_0\otimes \widetilde{\xi}_0^{\otimes 2n}}{\zeta}, j))=1
\]
and Lemma~\ref{lm:suspension} yields
\[
\Phi \frp \la Z(s_0), \relpcurve - Z(\zeta)- Z(\widetilde{\xi}_0), \frac{s_0\otimes \widetilde{\xi}_0^{\otimes 2n}}{\zeta}, j\ra \simA \sigma^{m-1}_{\relcurve} \circ r_0.
\]

Thus  $\sigma^{m}_{\relcurve} \circ r_1 - \sigma^{m}_{\relcurve} \circ r_0 \simA n h_\relcurve \circ (a_0-a_1)$.
\end{proof}

\begin{corollary} \label{cor:rigidity_lemma}
In the notation of Theorem~\ref{thm:main} let $\cF$ be a homotopy invariant stable linear framed presheaf over $S$. Suppose that either of the following holds.
\begin{enumerate}
\item
$2\in R^*$ and $nh\cF=0$;
\item
$2=0$ in $R$ and $n \cF=0$.
\end{enumerate}
Then $r_1^*=r_0^*\colon \cF(\relcurve) \to \cF(S)$.
\end{corollary}
\begin{proof}
(1) Theorem~\ref{thm:main} yields 
\[
r_1^*\circ (\sigma_\relcurve^m)^*= r_0^*\circ (\sigma_\relcurve^m)^* + i_1^*\circ H^* - i_0^*\circ H^* +  a^*\circ nh_\relcurve^*.
\]
Since $\cF$ is stable and homotopy invariant then $(\sigma_\relcurve^m)^*=\id$ and $i_0^*=(\pi_S^*)^{-1}=i_1^*$ with $\pi_{S}\colon \A^1\times S\to S$ being the projection. By the assumption we have $nh_\relcurve^*=0$. The claim follows.

(2) Analogous.
\end{proof}

\begin{theorem} \label{thm:rigidity}
Let $k$ be a field, $X$ be a smooth variety over $k$ and $x\in X$ be a closed point such that $k(x)/k$ is separable. Let $\cF$ be a homotopy invariant stable linear framed presheaf over $k$ and $n\in\NN$ be invertible in $k$. Suppose that either of the following holds.
\begin{enumerate}
\item
$\chark k\neq 2$ and $nh\cF=0$;
\item
$\chark k= 2$ and $n \cF=0$.
\end{enumerate}
Then the restriction to $\{x\}$ gives an isomorphism
\[
i_x^*\colon \cF(\Spec \struct_{X,x}^h) \xrightarrow{\simeq} \cF(\Spec k(x)).
\]
Here $\cF(\Spec \struct_{X,x}^h)=\varinjlim \cF(U_\alpha)$ with the limit taken along all the \etale neighborhoods of $x$ in $X$.
\end{theorem}
\begin{proof}
It is well-known that Corollary~\ref{cor:rigidity_lemma} yields the statement of the theorem via a geometric argument, see \cite[Proof of Theorem~2]{Gab92} or \cite[Proof of Theorem~4.4]{SV96}. We provide the argument below for the sake of completeness of the exposition. 

There is a canonical isomorphism $\cF(\Spec \struct_{X,x}^h)\cong \cF(\Spec \struct_{\A^d_{k(x)},0}^h)$ where $d=\dim X$. Every presheaf on $\ZF_*(k)$ is a presheaf on $\ZF_*(k(x))$ by means of the obvious functor $\ZF_*(k(x)) \to \ZF_*(k)$. Thus we may assume that $x$ is a rational point and that $(X,x)=(\A^d_k,0)$. 

Set $S_d=\Spec \struct_{\A^d_k,0}^h$. We have $\pi_d^*\circ i_d^*= \id$ for the projection $\pi_d\colon S_d\to \Spec k$ and the inclusion to the origin $i_d\colon \Spec k \to S_d$. Thus $\pi_d^*\colon \cF(\Spec k)\to \cF(S_d)$ is injective. We argue that $\pi_d^*$ is surjective by induction on $d$. The case of $d=0$ is trivial. Since $\cF(S_d)=\varinjlim \cF(U_\alpha)$ where $U_\alpha\to \A^d_k$ are the \etale neighborhoods of $0$ in $\A^d_k$ it is sufficient to show that for every $\alpha$ the image of $\cF(U_\alpha)$ in $\cF(S_d)$ belongs to $\pi_d^*(\cF(\Spec k))$.

Let $g_\alpha\colon U_\alpha \to \A^d_k$ be an \etale neighborhood of $\{0\}$. Applying Lemma~\ref{lm:construct_compactification} we obtain the following diagram.
\[
\xymatrix @C=3pc{
S_d \ar[r]_(0.6){g} \ar@/^1.0pc/[rrr]^{j_\alpha} & \cU \ar[rr]_{f} \ar[d]^{\rho} & & U_\alpha \\
& S_{d-1} \ar[ul]^{r_d} & &
}
\]
Here $j_\alpha\colon S_d\to U_\alpha$ is the structure morphism, $r_d$ is induced by the embedding $\A^{d-1}_k\times \{0\}\subset \A^{d}_k$, $\rho\colon \cU\to S_{d-1}$ is a smooth morphism of relative dimension $1$ and $\cU$ admits a fine compactification over $S_{d-1}$. Consider the following Cartesian square.
\[
\xymatrix{
\relcurve \ar[d] \ar[r] \podr & \cU \ar[d]^\rho \\
S_d \ar[r]^(0.45){\rho\circ g} & S_{d-1}
}
\]
The morphism $\relcurve\to S_d$ is smooth of relative dimension $1$ and $\relcurve$ admits a fine compactification over $S_d$. Let $r_1,r_2\colon S_d\to \relcurve$ be the morphisms given by $(\id_{S_d},g)$ and $(\id_{S_d}, g\circ r_d\circ \rho \circ g)$ respectively.

Applying Lemma~\ref{lm:extension} we extend $\cF$ to $\ZF_*(S_d)$. Corollary~\ref{cor:rigidity_lemma} yields
\[
r_1^*=r_2^*\colon \cF(\relcurve) \to \cF(S_d),
\]
whence
\[
g^*= g^*\circ \rho^* \circ r_d^* \circ g^*\colon \cF(\cU)\to \cF(S_d).
\]
Composing the above with $f^*$ we obtain
\[
j_\alpha^*=g^*\circ \rho^* \circ r_d^* \circ g^* \circ f^* \colon \cF(U_\alpha)\to \cF(S_d).
\]
Thus $\operatorname{Im}(j_\alpha^*)\subset \operatorname{Im}((g\circ \rho)^*)$. By the induction assumption we have $\cF(S_{d-1})=\operatorname{Im}  (\pi_{d-1}^*)$. Since $(g\circ \rho)^*(\operatorname{Im} (\pi_{d-1}^*))= \operatorname{Im} (\pi_{d}^*)$ then $\operatorname{Im}(j_\alpha^*)\subset  \operatorname{Im}(\pi_d^*)$ and the claim follows.
\end{proof}

\section{Rigidity for representable cohomology theories} \label{sect:motivic}

\begin{definition} \label{def:PT}
Let $\T=\A^1/(\A^1-\{0\})$ be the Morel-Voevodsky object considered as a pointed Nisnevich sheaf of sets and put $\PP^1=(\PP^1,\infty)$ for the pointed projective line also considered as a pointed Nisnevich sheaf of sets.

Let $S$ be a scheme and denote $\Sm_S^{\PP,\T}$ the category with objects being smooth schemes over $S$ and morphisms given by
\[
\Hom_{\Sm_S^{\PP,\T}}(X,Y)=\bigvee_{m\ge 0} \Hom_{{\mathrm{Shv_\Nis}}_\bullet}(X_+\wedge (\PP^1)^{\wedge m} ,Y_+\wedge \T^{\wedge m}).
\]
Here on the right-hand side we consider the wedge sum of the pointed sets of morphisms of pointed Nisnevich sheaves. The sets are pointed by the constant morphisms. For 
\[
f\in \Hom_{{\mathrm{Shv_\Nis}}_\bullet}(X_+\wedge (\PP^1)^{\wedge n},Y_+\wedge \T^{\wedge n}),\quad g\in \Hom_{{\mathrm{Shv_\Nis}}_\bullet}(Y_+\wedge (\PP^1)^{\wedge m} ,V_+\wedge \T^{\wedge m})
\]
the composition 
\[
g\circ f \in \Hom_{{\mathrm{Shv_\Nis}}_\bullet}(X_+\wedge (\PP^1)^{\wedge n+m}, V_+\wedge \T^{\wedge n+m})
\]
is defined as
\begin{multline*}
X_+ \wedge (\PP^1)^{\wedge n}\wedge (\PP^1)^{\wedge m} \xrightarrow{f\wedge \id} 
Y_+ \wedge \T^{\wedge n} \wedge (\PP^1)^{\wedge m} \xrightarrow{\tau} \\ 
\xrightarrow{\tau} Y_+ \wedge (\PP^1)^{\wedge m} \wedge \T^{\wedge n} \xrightarrow{g\wedge \id} 
V_+ \wedge \T^{\wedge m} \wedge \T^{\wedge n} \xrightarrow{\tau} 
V_+ \wedge \T^{\wedge n} \wedge \T^{\wedge m} ,
\end{multline*}
where $\tau$ stands for the corresponding permutation isomorphisms.
\end{definition}

\begin{definition} \label{def:FrPT}
Let $S$ be a scheme and $X,Y$ be schemes over $S$. An explicit framed correspondence $\Phi=(Z,U,\phi,g)\in \Fr_m^S(X,Y)$ gives rise to a morphism of pointed Nisnevich sheaves of sets
\[
\Xi(\Phi)\colon X_+ \wedge (\PP^1)^{\wedge m}  \to Y_+\wedge \T^{\wedge m}
\]
in the following way. Consider the commutative diagram
\[
\xymatrix{
U - Z \ar[r] \ar[d] & U \ar@/^1.5pc/[ddr]^{(g,\phi)} \ar[d]^p  & \\
 X \times (\Proj^1)^{\times m} - Z \ar[r]^j \ar@/_1.5pc/[drr]^{q} &  X\times (\Proj^1)^{\times m} & \\
 & & Y_+\wedge \T^{\wedge m} \cong (Y\times \A^n)/(Y\times (\A^m-0)).
}
\]
Here the square is Cartesian, $p$ is given by the composition $U\to X\times \A^m \to X\times (\Proj^1)^{\times m}$ for the standard immersion $\A^1=\Proj^1-\infty\subset \Proj^1$, $j$ is the open immersion and $q$ is the constant morphism that maps $X \times (\Proj^1)^{\times m} - Z$ to the distinguished point. The square is a Nisnevich cover, thus we have a morphism of Nisnevich sheaves
\[
(X \times (\Proj^1)^{\times m})_+ \to Y_+ \wedge \T^{\wedge m}
\]
that induces the desired morphism
\[
\Xi(\Phi)\colon X_+\wedge (\PP^1)^{\wedge m}  \to Y_+\wedge \T^{\wedge m}.
\]
This construction yields a map
\[
\Xi\colon \Fr_m^S(X,Y)\to \Hom_{{\mathrm{Shv_\Nis}}_\bullet}(X_+\wedge (\PP^1)^{\wedge m} ,Y_+\wedge \T^{\wedge m})
\]
and gives rise to a functor $\Xi$ fitting in the following commutative diagram
\[
\xymatrix{
\Sm_S \ar[d] \ar[dr] &  \\
\Fr_*(S) \ar[r]^(0.4)\Xi & \Sm_S^{\PP,\T}
}
\]
where the unlabeled functors are given by the obvious embeddings to the degree $0$.
\end{definition}

\begin{remark}
One can show \cite[Lemma~5.2]{GP14} that in the case of $S=\Spec k$ for a field $k$ the functor $\Xi$ is an equivalence of categories.
\end{remark}

\begin{definition} \label{def:hsh}
Let $S$ be a scheme. A \textit{$\PP^1$-spectrum} is a sequence $(E_0,E_1,\hdots)$ of pointed presheaves of simplicial sets on $\Sm_S$ together with the bonding morphisms $E_m\wedge \PP^1\to E_{m+1}$. A morphism of $\PP^1$-spectra is a sequence of morphisms of pointed presheaves respecting the bonding maps. Inverting the stable motivic equivalences as in \cite{Jar00} one obtains the \textit{motivic stable homotopy category} $\SH(S)$. See \cite{V98,MV99,Mor04} as an introduction to the motivic homotopy theory and as a reference for the basic properties that we use below.

Every pointed presheaf of simplicial sets $M$ gives rise to the \textit{suspension spectrum}
\[
\Sigma^\infty_{\PP^1} M = (M,M\wedge \PP^1, M\wedge (\PP^1)^{\wedge 2},\hdots )
\]
with the bonding maps being the identities. In particular, every smooth scheme $X$ over $S$ can be regarded as a representable presheaf of sets thus giving rise to the suspension spectrum $\Sigma^\infty_{\PP^1} X_+$. Denote $\SSp=\Sigma_{\PP^1}^{\infty} S_+$ the \textit{sphere spectrum}.

The category $\SH(S)$ is triangulated and monoidal, in particular, the sets of morphisms are modules over the ring $\pi_{0,0}^{\A^1}(\SSp)=\Hom_{\SH(S)}(\SSp,\SSp)$. The suspension functor $\Sigma_{\PP^1}\colon \SH(S)\to \SH(S)$, $(E_0,E_1,\hdots)\mapsto (\PP^1\wedge E_0,\PP^1\wedge E_1,\hdots)$, is invertible. We fix the following notation:
\[
\varepsilon = \Sigma^{-1}_{\PP^1}\Sigma^{\infty}_{\PP^1} \tau \in \pi_{0,0}^{\A^1}(\SSp),\quad h^{\SH}=\id_{\SSp}+\varepsilon,
\]
where $\tau\colon \PP^1_S\to \PP^1_S$ is given by $[x:y]\mapsto [-x:y]$. 

For $E\in \SH(S)$ and $p,q\in\Z$ let $E^{p,q}(-)$ be the presheaf of abelian groups on $\Sm_S$ given by
\[
X\mapsto \Hom_{\SH(S)}(\Sigma^\infty_{\PP^1} X_+, \Sigma_{\PP^1}^{q}E[p-2q]).
\]
\end{definition}

\begin{remark}
In some papers on motivic homotopy theory one also denotes 
\[
\pi^{\A^1}_{-p,-q}(E)(-)=E^{p,q}(-)
\]
and refers to it as the $(-p,-q)$-th \textit{$\A^1$-homotopy presheaf} of $E$.
\end{remark}

\begin{definition} \label{def:sigma_tilde}
Let $\gamma\colon \PP^1 \to \T$ be the morphism of pointed Nisnevich sheaves induced by the contraction $\PP^1\to\PP^1/(\PP^1-\{0\})$ composed with the (excision) isomorphism of Nisnevich sheaves $\T\cong \PP^1/(\PP^1-\{0\})$. The morphism 
\[
\Sigma^{\infty}_{\PP^1}\gamma \colon \Sigma^{\infty}_{\PP^1} \PP^1 \to \Sigma^{\infty}_{\PP^1} \T
\]
is invertible. Abusing the notation we write $\Sigma^{\infty}_{\PP^1}\gamma^{-1} = (\Sigma^{\infty}_{\PP^1}\gamma)^{-1}$.

Let $S$ be a scheme, $X,Y$ be smooth schemes over $S$ and consider
\[
f\in \Hom_{{\mathrm{Shv_\Nis}}_\bullet}(X_+\wedge (\PP^1)^{\wedge m} ,Y_+\wedge \T^{\wedge m}).
\]
Set 
\[
\widetilde{\Sigma}^\infty_{\PP^1} f = \Sigma_{\PP^1}^{-m} \left(\left(\id_{ \Sigma^\infty_{\PP^1}Y_+} \wedge (\Sigma^{\infty}_{\PP^1}\gamma^{-1})^{\wedge m}\right)  \circ \Sigma^{\infty}_{\PP^1}f \right)\in \Hom_{\SH(S)}(\Sigma^\infty_{\PP^1}X_+, \Sigma^\infty_{\PP^1}Y_+).
\]
\end{definition}

\begin{remark}
Roughly speaking, up to the identification of $\PP^1$ with $\T$, the morphism $\widetilde{\Sigma}^\infty_{\PP^1} f$ comes from the morphism of spectra that has the suspensions of $f$ starting from the stage $m$.
\end{remark}

\begin{lemma} \label{lm:fr_transfers}
The construction from Definition~\ref{def:sigma_tilde} gives rise to a functor $\widetilde{\Sigma}^{\infty}_{\PP^1}\colon \Sm_S^{\PP,\T} \to \SH(S)$ fitting in the commutative diagram
\[
\xymatrix{
\Sm_S \ar[d] \ar[rrd]^{\Sigma^\infty_{\PP^1}}& & \\
\Fr_*(S) \ar[r]_{\Xi} & \Sm_S^{\PP,\T} \ar[r]_{\widetilde{\Sigma}^{\infty}_{\PP^1}} & \SH(S).
}
\]
Moreover, for every $Y$ smooth over $S$ one has
\[
(\widetilde{\Sigma}^{\infty}_{\PP^1}\circ \Xi) (\sigma_Y)=\id_{\Sigma^{\infty}_{\PP^1} Y_+},\quad (\widetilde{\Sigma}^{\infty}_{\PP^1}\circ \Xi) (h_Y) = h^{\SH} \id_{\Sigma^{\infty}_{\PP^1} Y_+}.
\]
\end{lemma}
\begin{proof}
Straightforward.
\end{proof}

\begin{lemma}\label{lm:coh_transfers}
For $E\in\SH(S)$ and $p,q\in \Z$ there exists a canonical structure of a homotopy invariant stable linear framed presheaf on $E^{p,q}(-)$, i.e. there exists a canonical homotopy invariant stable linear framed presheaf $E^{p,q}_{fr}$ over $S$ fitting in the following commutative diagram.
\[
\xymatrix @C=5pc{
\Sm_S \ar[d] \ar[r]^{\Sigma^\infty_{\PP^1}}& \SH(S) \ar[d]^{\Hom_{\SH(S)}(-,\Sigma^{q}_{\PP^1}E[p-2q])}\\
\ZF_*(S) \ar[r]^{E^{p,q}_{fr}(-)} &  \mathrm{Ab}
}
\]
Moreover, if for some $n\in \mathbb{N}$ one has $nh^{\SH}E=0$ (or $nE=0$) then $nhE^{p,q}_{fr}=0$ (resp. $nE^{p,q}_{fr}=0$).
\end{lemma}
\begin{proof}
$E^{p,q}$ is homotopy invariant and $E^{p,q}(X\sqcup X)=E^{p,q}(X)\oplus E^{p,q}(X)$, thus the claim follows from Lemmas~\ref{lm:additivity} and~\ref{lm:fr_transfers}.
\end{proof}

\begin{theorem} \label{thm:rigidity_motivic}
Let $k$ be a field, $X$ be a smooth variety over $k$ and $x\in X$ be a closed point such that $k(x)/k$ is separable. Let $E\in \SH(k)$ and $n\in\NN$ be invertible in $k$. Suppose that either of the following holds.
\begin{enumerate}
\item
$\chark k\neq 2$ and $nh^{\SH}E=0$;
\item
$\chark k= 2$ and $n E=0$.
\end{enumerate}
Then for $p,q\in \Z$ the restriction to $\{x\}$ gives an isomorphism
\[
i_x^*\colon E^{p,q}(\Spec \struct_{X,x}^h) \xrightarrow{\simeq} E^{p,q}(\Spec k(x)).
\]
Here $E^{p,q}(\Spec \struct_{X,x}^h)=\varinjlim E^{p,q}(U_\alpha)$ with the limit taken along the \etale neighborhoods of $x$ in $X$.
\end{theorem}
\begin{proof}
Follows from Lemma~\ref{lm:coh_transfers} and Theorem~\ref{thm:rigidity}.
\end{proof}

If the base field is perfect then Morel's computation $\Hom_{\SH(k)}(\SSp,\SSp)\cong \GW(k)$ (\cite[Theorem~6.4.1]{Mor04} and \cite[Corollary~6.43]{Mor12}) gives the following reformulation of Theorem~\ref{thm:rigidity_motivic}. The statement was brought to our attention by Tom Bachmann.
\begin{corollary} \label{cor:perfect_rigidity}
	Let $k$ be a perfect field, $X$ be a smooth variety over $k$ and $x\in X$ be a closed point. Let $E\in \SH(k)$ and suppose that $\phi E=0$ for some $\phi\in\GW(k)\cong \Hom_{\SH(k)}(\SSp,\SSp)$ such that $\rank \phi$ is invertible in $k$. Then for $p,q\in \Z$ the restriction to $\{x\}$ gives an isomorphism
	\[
	i_x^*\colon E^{p,q}(\Spec \struct_{X,x}^h) \xrightarrow{\simeq} E^{p,q}(\Spec k(x)).
	\]
	Here $E^{p,q}(\Spec \struct_{X,x}^h)=\varinjlim E^{p,q}(U_\alpha)$ with the limit taken along the \etale neighborhoods of $x$ in $X$.
\end{corollary}
\begin{proof}
	Choose $\phi \in \GW(k)$ such that $\phi E=0$ and $\rank \phi$ is invertible in $k$. Let $h\in \GW(k)$ be the hyperbolic form of rank $2$. One has $\phi h= n h$ for $n=\rank \phi$ and in the case of $\chark k\neq 2$ the claim follows from Theorem~\ref{thm:rigidity_motivic} since $nhE=h\phi E=0$. 
	
	Let $\chark k= 2$. Recall that in the Witt ring $W(k)$ every element can be represented by a diagonal form whence for every element in $W(k)$ its square is either $0$ or $1$. Since the rank of $\phi$ is odd then $\phi^2=1$ in $W(k)$. Then $\phi^2 = 1+ mh$ with $m=\tfrac{(\rank \phi)^2-1}{2}$ whence
	\[
	(1+2m-mh)\phi^2 = 1+2m
	\]
	and the claim follows from Theorem~\ref{thm:rigidity_motivic} since $(1+2m)E=(1+2m-hm)\phi^2 E=0$.
\end{proof}

\end{document}